\documentclass{article}

\usepackage{tikz}
\usepackage[usenames,dvipsnames]{pstricks}
\usepackage{epsfig}
\usepackage{pst-grad} 
\usepackage{pst-plot} 
\usepackage[space]{grffile} 
\usepackage{etoolbox} 
\makeatletter 
\patchcmd\Gread@eps{\@inputcheck#1 }{\@inputcheck"#1"\relax}{}{}
\makeatother

\usepackage[english]{babel}

\usepackage{geometry}
\usepackage{amsmath,amsfonts,amssymb,amsthm}
\usepackage{graphicx}
\usepackage[colorlinks=true, allcolors=blue]{hyperref}
\usepackage{comment}
\usepackage{breqn}
\usepackage{caption}
\usepackage{subcaption}
\usepackage{csquotes}

\usepackage[maxbibnames=99,backend=bibtex,style=numeric]{biblatex}
\newtheorem{definition}{Definition}
\newtheorem{theorem}[definition]{Theorem}
\newtheorem{claim}{Claim}[definition]
\newtheorem{proposition}[definition]{Proposition}
\newtheorem{lemma}[definition]{Lemma}
\newtheorem{corollary}[definition]{Corollary}

\newcommand{\linesq}{\mathcal{L}_\mathcal{Q}}
\newcommand{\betq}{\mathcal{B}}
\newcommand{\qq}{\mathcal{Q}}
\newcommand{\ordx}{\prec_X}
\newcommand{\ordy}{\prec_Y}
\newcommand{\ordz}{\prec_Z}
\newcommand{\ordw}{\prec_W}
\newcommand{\ordv}{\prec_V}

\title{Quasimetric spaces with few lines}
\author{
Guillermo Gamboa Quintero \thanks{Computer Science Institute of Charles University, Prague, Czechia. Supported by GA\v{C}R grant 22-17398S. \textit{E-mail}: gamboa@iuuk.mff.cuni.cz} 
\and 
Mart\'{i}n Matamala \thanks{DIM-CMM, CNRS-IRL2807, Universidad de Chile, Chile. Supported by ANID Basal program FB210005.  \\ \textit{E-mail}: mar.mat.vas@dim.uchile.cl} 
\and 
Juan Pablo Pe\~{n}a \thanks{Departamento de Ingenier\'{i}a Matem\'{a}tica, Universidad de Chile, Chile. Supported by ANID Doctoral Fellowship grant 21211955. \textit{E-mail}: juan.pena@dim.uchile.cl}}

\date{}

\begin{document}
\maketitle

\begin{abstract}
Chen and Chvátal conjectured in 2008 that in any finite metric space either there is a line containing all the points –a universal line –, or the number of lines is at least the number of points. This is a generalization of a classical result due to Erd\H{o}s that says that a set of $n$ non-collinear points in the Euclidean plane defines at least $n$ different lines.

A line of a metric space with metric $\rho$ is defined in terms of a notion called the betweenness of the space which is the set of all triples $(x,z,y)$ such that $\rho(x,y)=\rho(x,z)+\rho(z,y)$.

In this work we prove that for each $n\geq 4$ there are $p_3(n)$ non isomorphic betweennesses arising from \emph{quasimetric} spaces with $n$ points, without universal lines and with exactly 3 lines, where $p_3(n)$ is the number of partitions of an integer $n$ into three parts.
We also prove that for $n\geq 5$, there are $2p_3(n-1)$ non isomorphic betweennesses arising from quasimetric spaces on $n$ points, without universal lines and with exactly 4 lines. Here two betweennesses are isomorphic if they are isomorphic as relational structures.

None of the betweennesses mentioned above is metric which implies that Chen and Chvátal's conjecture is valid for metric spaces with at most five points.

\end{abstract}

\section{Introduction}

The De Bruijn-Erd\H{o}s Theorem \cite{dbe} is an important result in combinatorics. It states that if $E$ is a collection of subsets of a ground set $V$ of $n$ points such that for all $e \in E$ it holds that $2 \leq \lvert e \rvert \leq n-1$ and for every pair $x,y \in V$ there is exactly one member $e \in E$ such that $x,y \in e$, then $\lvert E \rvert \geq n$. This theorem  captures a combinatorial property of the plane stating that the number of distinct lines defined by $n$ noncollinear points is at least $n$. A direct proof of this latter result was presented by Erd\H{o}s \cite{erdos} using the Sylvester-Gallai Theorem  \cite{sylvester,gallai} as part of an induction argument.

In \cite{chenchvatal}, Chen and Chv\'{a}tal generalized the notion of line to metric spaces in the following way: For a metric space $(V, \rho)$ and distinct points $x,y,z \in V$, we say $z$ is \textit{between} $x$ and $y$ if $\rho(x,y) = \rho(x,z) + \rho(z,y)$ and denoted by $[xzy]$. The \textit{line $\overline{xy}$ defined by $x$ and $y$} is the set of points $z \in V$ such that $[zxy], [xzy]$ or $[xyz]$ holds (In itself, this ternary relation called \textit{metric betweenness} was first studied by Menger \cite{menger} and, more recently by Chv\'{a}tal \cite{chvatal2004}). A \textit{universal} line is a line containing all the points of the space. Naturally, Chen and Chv\'{a}tal \cite{chenchvatal} presented the following question as a generalization of the De-Bruijn-Erd\H{o}s theorem for metric spaces: Is it true that every metric space $(V, \rho)$ on $n$ points with no universal line has at least $n$ distinct lines? This question is referred to as the Chen-Chv\'{a}tal conjecture and it remains open to this day, although several asymptotic lower bounds have been presented for the number of lines in metric spaces without universal lines. The initial lower bound for the number of lines in a metric space on $n$ points with no universal line was $\Omega( \log n )$ lines \cite{chenchvatal} and it was later improved to $\Omega(n^{2/7})$ lines in \cite{chichv}, then to $\Omega(\sqrt{n})$ lines in \cite{aboulker} and, most recently to $\Omega(n^{2/3})$ lines in \cite{huang}. The latter remains the best asymptotic lower bound, up to date. Other results concerning the Chen-Chv\'{a}tal conjecture for specific types of metric spaces can be found in \cite{abbemaza, abkha, beboch, bekharo, ch2014, kantor, jualrove}

While the Chen-Chv\'{a}tal conjecture is stated in the context of metric spaces, it is natural to extend it to objects that share similar properties with these. In \cite{aramat}, Araujo-Pardo and Matamala first considered the Chen-Chv\'{a}tal conjecture in the context of quasimetric spaces, which are spaces of points equipped with a function which has the same properties as a metric with the exception that it is not necessarily symmetric. The quasimetric spaces studied were those arising from tournaments and bipartite tournaments, and they were shown to satisfy the Chen-Chv\'{a}tal conjecture. Later on, Araujo-Pardo, Matamala and Zamora \cite{aramatzam} showed that there is a quasimetric space on four points with three lines, none of which are universal. The betweenness arising from such quasimetric space is such that any quasimetric space on four points without universal line and with exactly three lines defines a betweenness isomorphic to it, as relational structures. Moreover, this space is not that of a metric space.  As a consequence, the Chen-Chv\'{a}tal conjecture does not hold for quasimetric spaces, but does hold for metric spaces of four points. 

In this paper, we show that there is a substantial difference between quasimetric spaces and metric spaces when it comes to their lines. Recall that metric spaces on $n$ points with no universal line has $\Omega(n^{2/3})$ lines \cite{huang}. In constrast, we show that for every $n \geq 4$ there are quasimetric spaces on $n$ points with only $3$ lines, none of which are universal. Surprisingly, we prove that the betweennesses defined by these spaces are such that any quasimetric space on $n$ points without universal lines and with exactly three lines defines a betweenness isomorphic to one of them, as relational structures. 
We extend this result to the case of quasimetric spaces with exactly four lines, none of them being universal. We prove that for each $n\geq 5$ there are quasimetric spaces on $n$ points without universal lines and with exactly four lines. We can also prove that any quasimetric space on $n$ points with four lines, none of which are universal defines a betweenness which is isomorphic to one of them, as relational structures. The structure of this paper is the following: in section 2, we introduce some preliminary notation as well as some properties of quasimetric spaces with at most 4 non universal lines. In section 3, we present our construction for arbitrarily large quasimetric spaces with $k=3$ lines, none of which is universal, as well as presenting a classification of such spaces. Section 4 is analogous to Section 3 for $k=4$. Finally, we conclude with a discussion and possible questions to extend this work. 

\section{Preliminaries}

\begin{definition}
    A quasimetric space $\qq=(Q, \rho)$ is a duple, where $Q$ is a set of \textit{points} and $\rho : Q^2 \rightarrow [0, \infty)$ is a function called a \textit{quasimetric} satisfying
    \begin{itemize}
        \item $\rho(x,y) = 0 \Leftrightarrow x=y$.
        \item $\rho(x,y) \leq \rho(x,z) + \rho(z,y)$.
    \end{itemize}
    for all $x,y,z \in Q$. We say that $(Q, \rho)$ is of size $n$ if $\lvert Q \rvert=n$.
\end{definition}

Consider a quasimetric space $\qq=(Q, \rho)$ and let $x,y \in Q$. We define the \textit{segment} of $x$ and $y$, denoted by $[xy]$, as
$$[xy] := \{ z \in Q \  \vert \ \rho(x,y) = \rho(x,z) + \rho(z,y) \}.$$
Notice that $x,y\in [xy]$ as $\rho(x,x)=\rho(y,y)=0$. We define the line $\overrightarrow{xy}$ as
$$\overrightarrow{xy} := \{ z \in Q \ \vert \ x \in [zy] \vee z \in [xy] \vee y \in [xz] \}.$$
We denote by $\linesq$ the set of all lines of $\qq$. We will say that a line $\ell \in \linesq$ is \textit{universal} if $\ell = Q$. The following proposition gives us a lower bound on the number of lines in a quasimetric space with no universal line.

\begin{proposition}
    Any quasimetric space $\qq$ of size $n \geq 3$ with no universal line satisfies that $\lvert \linesq \rvert \geq 3$.
\end{proposition}

\begin{proof}
    Since there is no universal line and every pair of points belongs to at least one line, there are at least two distinct lines. Let us denote them by $\ell_1$ and $\ell_2$. Assume for a contradiction that there are no other lines in $\linesq$. Without loss of generality, we may assume that there is a point $x \in \ell_1 \setminus \ell_2$. Then, for any $y \in Q$, we have that $\overrightarrow{xy} = \overrightarrow{yx} = \ell_1$ so $y \in \ell_1$. So, the line $\ell_1$ is universal, which is a contradiction.
\end{proof}

Let $\qq=(Q, \rho)$ be a quasimetric space. The \textit{betweenness} of $\qq$, denoted by $\betq_{\qq}$, is the relation defined as $$(x,y,z)\in \betq_{\qq} \iff x,y,z \text{ are all distinct and } \rho(x,y)+\rho(y,z)=\rho(x,z).$$ Such a relation was first studied by Menger \cite{menger} in the context of metric spaces. Moreover, Chv\'{a}tal \cite{chvatal2004} presented an infinite set of axioms describing the betweenness relation arising from metric spaces and also presented a polynomial time algorithm to determine whether a betweenness arises from a metric space. For simplicity, we write $xyz$ to refer to the triple $(x,y,z)$ and we write $\betq$, instead of $\betq_{\qq}$, when $\qq$ is clear from the context. Note that $\betq$ satisfies two important properties that we will use throughout this paper:
\begin{align}
    xyz \in \betq &\Rightarrow yxz, xzy \notin \betq, \label{prop:1} \\
    wxy, wyz \in \betq &\iff wxz, xyz \in \betq \label{prop:2},
\end{align}
for all $x,y,z,w \in \betq$. We say that a sequence $x_1 x_2 \ldots x_m$ of pairwise distinct points of $Q$ is a \textit{geodesic} if for any $1 \leq i < j < k \leq m$ it holds that $x_i x_j x_k \in \betq$ and we say the set $\{ x_1, \ldots, x_m \}$ is \textit{geodetic}. This is analogous to the definition of geodesic applied to metric spaces presented in \cite{aboulker}. A geodesic $x_1 x_2 \ldots x_m$ is \textit{maximal} if there is no point $w$ such that at least one of the sequences $wx_1 \ldots x_m$, $x_1 x_2 \ldots x_mw$ or $x_1 \ldots x_i w x_{i+1} \ldots x_m$ is a geodesic for any $1 \leq i \leq m$. 

\begin{lemma}
    \label{l:linearsystem}
    For any four pairwise distinct points $w,x,y,z \in Q$ the following holds:
\begin{align}
    wxy,ywz,zyx,xzw \in \betq &\Rightarrow wzy,yxz,zwx,xyw \in \betq, \label{prop:4}\\
    wxy, yzx,xwz,zyw \in \betq &\Rightarrow wzy, ywx,xyz,zxw\in \betq, \label{prop:5} \\
    wxy, yzw,xwz,zyx \in \betq &\Rightarrow wzy, yxw,xyz,zwx \in \betq. \label{prop:6}
\end{align}
\end{lemma}

\begin{proof} 
    It suffices to prove (\ref{prop:4}) since the arguments used to prove (\ref{prop:5}) and (\ref{prop:6}) are analogous. Assume that $wxy,ywz,zyx,xzw \in \betq$, so
    \begin{equation}
        \label{eq:deltaone}
        \begin{alignedat}{2}
            \rho(w,x) &+ \rho(x,y) &&= \rho(w,y),\\
            \rho(y,w) &+ \rho(w,z) &&= \rho(y,z),\\
            \rho(z,y) &+ \rho(y,x) &&= \rho(z,x),\\
            \rho(x,z) &+ \rho(z,w) &&= \rho(x,w).
        \end{alignedat}
    \end{equation}
    Let $\Delta := \rho(w,z )+ \rho(z,y) + \rho(y,x) + \rho(x,z) + \rho(z,w) + \rho(w,x) + \rho(x,y) + \rho(y,w)$. By the triangle inequality, we have that
    \begin{equation*}
        \Delta \geq \rho(w,y)+\rho(y,z)+\rho(z,x)+\rho(x,w).
    \end{equation*}
    Also, by (\ref{eq:deltaone}) it follows that
    \begin{equation*}
    \Delta = \rho(w,y)+\rho(y,z)+\rho(z,x)+\rho(x,w).  
    \end{equation*}
    This means that
    \begin{equation*}
        \begin{alignedat}{2}
            \rho(w,z) &+ \rho(z,y) &&= \rho(w,y), \\
            \rho(y,x) &+ \rho(x,z) &&= \rho(y,z), \\
            \rho(z,w) &+ \rho(w,x) &&= \rho(z,x), \\
            \rho(x,y) &+ \rho(y,w) &&= \rho(x,w)
        \end{alignedat}
    \end{equation*}
    which is equivalent to $wzy, yxz, zwx, xyw \in \betq$.
\end{proof}

Let $\betq_{\mathcal{P}}$ and $\betq_{\qq}$ be the betweennesses of two quasimetric spaces $\qq = (Q, \rho)$ and $\mathcal{P} = (P, \varrho)$, respectively. 
The betweennesses $\betq_{\qq}$ and $\betq_{\mathcal{P}}$ are isomorphic as relational structures, denoted as $\betq_{\qq} \cong \betq_{\mathcal{P}}$,  when there is a bijection $\varphi: Q \rightarrow P$ such that the function assigning to the triple $xyz$ the triple $\varphi(x)\varphi(y)\varphi(z)$ defines a bijection from $\betq_\qq$ to $\betq_{\mathcal{P}}.$ In this work, to ease the presentation, we shall just say that $\varphi$ is an isomorphism between $\betq_\qq$ and $\betq_{\mathcal{P}}$, and that $\betq_\qq$ and $\betq_{\mathcal{P}}$ are isomorphic.

In \cite{aramatzam}, Araujo-Pardo, Matamala and Zamora proved that the betweennesses of any quasimetric spaces on four points with exactly three lines (none of which is universal) are the same up to isomorphism. In Sections \ref{s:three} and \ref{s:four} we expand this result for any quasimetric space without universal lines with three and four lines, respectively.

For a subset $S$ of $Q$ we denote by $\betq (S)$ the subset of $\betq$ composed by all the triples $xyz$ such that $\{x,y,z\}\subseteq S$, that is, $$\betq (S) = \{ xyz \mid x,y,z \in S \text{ and } xyz \in \betq \}.$$ A triple $x,y,z \in Q$ is \textit{collinear} if $\betq ( \{ x,y,z \} ) \neq \emptyset$. Note that $$x,y,z \text{ are collinear } \Leftrightarrow x \in \overrightarrow{yz} \cup \overrightarrow{zy} \Leftrightarrow y \in \overrightarrow{xz} \cup \overrightarrow{zx} \Leftrightarrow z \in \overrightarrow{xy} \cup \overrightarrow{yx}.$$ Also, we say that points $x,y \in Q$ are \textit{symmetric} if $\overrightarrow{xy} = \overrightarrow{yx}$ and we denote the line defined by two such symmetric points as $\overline{xy}$. A set $S$ of points is \emph{symmetric} if any two points of the set are pairwise symmetric and there is a line $\ell$ such that $\overline{xy} = \ell$, for all $x,y \in S$. Finally, we say a line $\ell$ is \textit{symmetric} if there are symmetric points $x,y \in \ell$ such that $\ell = \overline{xy}$.

For two symmetric points $x$ and $y$ we have that $z \notin \overline{xy}$ if and only if the triple $x,y,z$ is not collinear. Our next result proves that in a quasimetric space with at most four lines, none of which is universal, three distinct points define a collinear triple whenever at least two of them are nonsymmetric.

\begin{lemma}\label{l:covering}
    Let $x,y\in Q$ such that $\overrightarrow{xy}\neq \overrightarrow{yx}$. If $Q\neq \overrightarrow{xy}\cup \overrightarrow{yx}$, then $\qq$ has a universal line or $\lvert \linesq \rvert \geq 5$.
\end{lemma}

\begin{proof}
    Suppose, for a contradiction, that $Q\neq \overrightarrow{xy}\cup \overrightarrow{yx}$ and $\lvert \linesq \rvert \leq 4$, with no universal line. Let $u \in Q \setminus ( \overrightarrow{xy}\cup \overrightarrow{yx}$). Then, the set $\{\overrightarrow{ux},\overrightarrow{xu}\}\cap \{\overrightarrow{uy},\overrightarrow{yu}\}$ is empty and none of the lines $\overrightarrow{ux},\overrightarrow{xu},\overrightarrow{uy}$ or $\overrightarrow{yu}$ is an element of the set $\{\overrightarrow{xy},\overrightarrow{yx}\}$.  If any of the sets $\{ u,x \}$ or $\{ u,y \}$ are not symmetric, then $\lvert \linesq \rvert \geq 5$, which is a contradiction. So, the sets $\{ u,x \}$ and $\{ u,y \}$ are symmetric. If there is a point $z \in Q \setminus (\overline{ux} \cup \overline{uy})$, then the line $\overrightarrow{zu}\notin \{\overrightarrow{xy},\overrightarrow{yx},\overline{ux},\overline{uy}\}$, a contradiction. So, we have that $\linesq=\{\overrightarrow{xy},\overrightarrow{yx},\overline{ux},\overline{uy}\}$ and $Q=\overline{ux}\cup \overline{uy}$.  Without loss of generality, let $v\in \overrightarrow{xy}\setminus \overrightarrow{yx}$. 

    First, we prove that $v\in \overline{ux}\cap \overline{uy}$. In fact, if $v\notin \overline{ux}$, then $u\notin \overrightarrow{vx}\cup \overrightarrow{xv}$. So, $x$ and $v$ are symmetric and $\overline{xv}=\overrightarrow{xy}$, since $v\notin \overrightarrow{yx}$. This latter fact implies that none of the triples 
    $vyx,yvx,yxv$ belongs to $\betq$  and, since $y\in \overrightarrow{vx}$, then $vxy \in \betq$. This, together with the fact that $y\in \overrightarrow{xv}$, implies that $xyv \in \betq$. So, $v$ and $y$ are symmetric and with $\overline{vy}=\overrightarrow{xy}$ and $u \notin \overline{vy}$. As a consequence, we get that $v \notin \overline{uy}$. Hence, 
    $\overline{uv} \neq \overline{{uy}}$ which implies that 
    $\overline{uv} \notin \{ \overrightarrow{xy},\overrightarrow{yx},\overline{ux},\overline{uy} \}$, a contradiction. An analogous reasoning shows that $v\notin \overline{uy}$ implies $\overline{uv}\notin \{\overrightarrow{xy},\overrightarrow{yx},\overline{ux},\overline{uy}\}$.

    Now, since $v \in \overline{ux}\cap \overline{uy}$, then $u \in (\overrightarrow{vx} \cup \overrightarrow{xv}) \cap  (\overrightarrow{vy} \cup \overrightarrow{yv} )$. From $v\in \overrightarrow{xy}$ we also get that $x\in \overrightarrow{yv}\cup \overrightarrow{vy}$ and $y\in \overrightarrow{xv}\cup \overrightarrow{vx}$. Thus,
    \begin{equation}
        \label{eq:contra}
        \{\overrightarrow{xv},\overrightarrow{vx}\}=\{\overline{ux},\overrightarrow{xy}\} \textrm{ and }
    \{\overrightarrow{yv},\overrightarrow{vy}\}=\{\overline{uy},\overrightarrow{xy}\}.
    \end{equation}
    It is left to prove that (\ref{eq:contra}) implies that $\overrightarrow{vu}=\overrightarrow{uv}=\overline{ux}=\overline{uy}$, a contradiction. If $(\overrightarrow{vx},\overrightarrow{xv})=(\overline{ux},\overrightarrow{xy})$, then $vxu, uvx \in \betq$. Hence, $x \in \overrightarrow{vu}\cap \overrightarrow{uv}$ which implies that $u$ and $v$ are symmetric and $\overline{vu}=\overline{ux}$. Similarly, if $(\overrightarrow{xv},\overrightarrow{vx})=(\overline{ux},\overrightarrow{xy})$ then $uxv, xvu \in \betq$. So, $x\in \overrightarrow{vu} \cap \overrightarrow{uv}$, which implies that $u$ and $v$ are symmetric and $\overline{uv} =\overline{ux}$.  The same analysis applied to the equality $\{\overrightarrow{yv},\overrightarrow{vy}\}=\{\overline{uy},\overrightarrow{xy}\}$ shows that $\overline{vu}=\overline{uy}$.
\end{proof}

\begin{lemma}
    \label{l:uniquetriples}
    Let $x,y,z \in Q$ such that $x \notin \overrightarrow{yz}$ and $y \notin \overrightarrow{zx}$. If $x \in \overrightarrow{zy}$, $y \in \overrightarrow{xz}$ or $z \in \overrightarrow{xy}$, then $\betq(\{x,y,z\}) = \{ xzy \}$. Otherwise, it follows that $\betq(\{x,y,z\}) = \emptyset$.    
\end{lemma}

\begin{proof}
    Since $x \notin \overrightarrow{yz}$ then $xyz,yxz,yzx \notin \betq$. Similarly, since $y\notin \overrightarrow{zx}$ then $yzx,zyx,zxy \notin \betq$. So, $x \in \overrightarrow{zy}$, $y \in \overrightarrow{xz}$ or $z \in \overrightarrow{xy}$ holds if and only if $xzy \in \betq$.
\end{proof}

\begin{lemma}
\label{l:nouniquetriples}
    Let $x,y,z \in Q$ such that $\betq(\{x,y,z\})=\{ xzy \}$. Then, $\qq$ has a universal line or $\lvert \linesq \rvert \geq 5$.
\end{lemma}

\begin{proof}
    Suppose, for a contradiction, that $\qq$ has at most four lines, none of which is universal. Since $\betq(\{x,y,z\})=\{ xzy \}$, then the lines $\overrightarrow{zx}, \overrightarrow{yz}, \overrightarrow{yx}, \overrightarrow{xz}$ are pairwise distinct so $\lvert \linesq \rvert = 4$. More explicitly, we have that $y \notin \overrightarrow{zx}$, $x \notin \overrightarrow{yz}$, $z \notin \overrightarrow{yx}$ and $\overrightarrow{xz} = \overrightarrow{zy} = \overrightarrow{xy}$. Now, since $\overrightarrow{xz}$ is not universal, then the set $Q \setminus \overrightarrow{xz}$ is not empty so let $w \in Q \setminus \overrightarrow{xz}$. It is clear that 
    
    $$\{\overrightarrow{xw},\overrightarrow{wx}\}=\{\overrightarrow{zx},\overrightarrow{yx}\}, \{\overrightarrow{yw},\overrightarrow{wy}\}=\{\overrightarrow{yx},\overrightarrow{yz}\} \textrm{ and } \{\overrightarrow{zw},\overrightarrow{wz}\}=\{\overrightarrow{zx},\overrightarrow{yz}\}.$$
    The equalities follow from the fact that for any pairwise distinct $a,b,c \in \{ w,x,y,z \}$ we have that $\betq(\{ a,b,c \}) \neq \emptyset$, by Lemma \ref{l:covering}.

    First, suppose that $( \overrightarrow{xw} , \overrightarrow{wx} ) = ( \overrightarrow{yx} , \overrightarrow{zx} )$. Since $w\notin \overrightarrow{xy}$, $y\notin \overrightarrow{wx}$ and $w\in \overrightarrow{yx}$, by Lemma \ref{l:uniquetriples} we get that $\betq ( \{w,x,y\} ) = \{ yxw \}$. This implies that $( \overrightarrow{yw} , \overrightarrow{wy} ) = ( \overrightarrow{yx} , \overrightarrow{yz} )$. Also, since $w \notin  \overrightarrow{zy}$, $z \notin \overrightarrow{yw}$ and $w \in \overrightarrow{yz}$, by Lemma \ref{l:uniquetriples} we get that $\betq( \{ w,y,z \} ) = \{ wyz \}$. So, $( \overrightarrow{zw} , \overrightarrow{wz} ) = ( \overrightarrow{zx} , \overrightarrow{yz} )$. Finally, since $z \notin \overrightarrow{xw}$, $x\notin \overrightarrow{wz}$ and $w\in \overrightarrow{zx}$, applying Lemma \ref{l:uniquetriples} gives us that $\betq( \{ w,x,z \} ) = \{ zwx \}$. Therefore, we get that $\betq( \{ w,x,y,z \} )=\{ xzy,yxw,wyz,zwx \}$ and, by Lemma \ref{l:linearsystem}, this is a contradiction.

    Now, suppose that $(\overrightarrow{xw},\overrightarrow{wx})=(\overrightarrow{zx},\overrightarrow{yx})$. By Lemma \ref{l:uniquetriples}, since $z\notin \overrightarrow{wx}$, $w\notin \overrightarrow{xz}$ and $w\in \overrightarrow{zx}$, then $\betq(\{w,x,z\})=\{zxw\}$. This implies that $(\overrightarrow{zw},\overrightarrow{wz})=(\overrightarrow{zx},\overrightarrow{yz})$. We first consider the case that $(\overrightarrow{yw},\overrightarrow{wy})=(\overrightarrow{yx},\overrightarrow{yz})$. By Lemma \ref{l:uniquetriples}, since $y\notin \overrightarrow{xw}$, $x\notin \overrightarrow{wy}$ and $w\in \overrightarrow{yx}$, then $\betq(\{w,x,y\})=\{ywx\}$. Also, by Lemma \ref{l:uniquetriples}, since $w\notin \overrightarrow{zy}$, $z\notin \overrightarrow{yw}$ and $w\in \overrightarrow{yz}$ then $\betq(\{w,y,z\})=\{wyz\}$. Thus, we get that $\betq(\{w,x,y,z\})=\{xzy,ywx,wyz,zxw\}$, a contradiction by Lemma \ref{l:linearsystem}. The remaining case is $(\overrightarrow{yw},\overrightarrow{wy})=(\overrightarrow{yz},\overrightarrow{yx})$. Again, by Lemma \ref{l:uniquetriples}, since $y\notin \overrightarrow{zw}$, $z\notin \overrightarrow{wy}$ and $w\in \overrightarrow{yz}$ then $\betq(\{w,y,z\})=\{ywz\}$. Also, by Lemma \ref{l:uniquetriples}, since $w\notin \overrightarrow{xy}$, $x\notin \overrightarrow{yw}$ and $w\in \overrightarrow{yx}$, then $\betq(\{w,x,y\})=\{wyx\}$. Therefore, we conclude that $\betq(\{w,x,y,z\})=\{xzy,ywz,wyx,zxw\}$, which is a contradiction by Lemma \ref{l:linearsystem}.
\end{proof}

More often, we will use Lemma \ref{l:nouniquetriples} as follows. If $\qq$ is a quasimetric space with less that five lines none of them being universal, then the betweenness induced by any three points of $\qq$ must contain at least two triples. In particular, if there are distinct points $x,y$ and $z$ such that $z\in \overrightarrow{xy}\setminus \overrightarrow{yx}$, then we have that 
$zyx,yzx$ and $yxz$ do not belong to $\betq$ and that at least two of the triples $zxy,xzy$ or $xyz$ must belong to $\betq$. But $xzy\in \betq$ implies that $zxy$ and $xyz$ do not belong to $\betq$. Hence, the only possibility is that $zxy$ and $xyz$ belong to $\betq$.

\begin{lemma}
    \label{l:nonsymmetric}
    Let $\qq$ be a quasimetric space such that no two points are symmetric. Then, $\qq$ has a universal line or $\lvert \linesq \rvert \geq 5$.
\end{lemma}

\begin{proof}
    Suppose, for a contradiction, that $\qq$ has at most four lines, none of which is universal. Let $x,y\in Q$, so $\overrightarrow{xy} \neq \overrightarrow{yx}$. It follows that $Q=\overrightarrow{xy}\cup \overrightarrow{yx}$, by Lemma \ref{l:covering}. Since $\qq$ has no universal line, then there are points $z\in \overrightarrow{xy} \setminus \overrightarrow{yx}$ and $w \in \overrightarrow{yx} \setminus \overrightarrow{xy}$. By Lemma \ref{l:nouniquetriples}, the latter implies that $wyx,yxw\in \betq$ and $y\in \overrightarrow{wx}\cap \overrightarrow{xw}$. So, $\qq$ has four lines and $\linesq=\{\overrightarrow{xy},\overrightarrow{yx},\overrightarrow{wz},\overrightarrow{zw}\}$.

    Moreover, there are $u\in \overrightarrow{wz}\setminus \overrightarrow{zw}$ and $v\in \overrightarrow{zw}\setminus\overrightarrow{wz}$, which implies that $\{\overrightarrow{xy},\overrightarrow{yx}\}=\{\overrightarrow{uv},\overrightarrow{vu}\}.$ Hence, we may assume that $x\notin \overrightarrow{wz}$ and $y\notin\overrightarrow{zw}$. Now, since $y\in \overrightarrow{wx}\cap \overrightarrow{xw}$, $x\notin \overrightarrow{wz}$, $w\notin \overrightarrow{xy}$ and $y\notin \overrightarrow{zw}$ we get that $\overrightarrow{wx},\overrightarrow{xw}\notin\{\overrightarrow{xy},\overrightarrow{zw},\overrightarrow{wz}\},$ which implies that $w$ and $x$ are symmetric, a contradiction.
\end{proof}

\begin{lemma}
    \label{l:threeinternonempty}
    Let $x,y,z \in Q$ such that $z\notin \overrightarrow{xy}\cup \overrightarrow{yx}$ and $(\overrightarrow{xy}\cup \overrightarrow{yx})\cap (\overrightarrow{xz}\cup \overrightarrow{zx})\cap(\overrightarrow{yz}\cup \overrightarrow{zy})$ is non-empty, then $\qq$ has a universal line or $\lvert \linesq \rvert \geq 5$.
\end{lemma}

\begin{proof}
    Suppose, for a contradiction, that $\qq$ has at most four lines, none of which is universal. Since $z\notin \overrightarrow{xy}\cup \overrightarrow{yx}$, then $x\notin \overrightarrow{zy}\cup \overrightarrow{yz}$ and $y\notin \overrightarrow{xz}\cup \overrightarrow{zx}$. By Lemma \ref{l:covering}, the pair of points $x$ and $y$ is symmetric as well as the pairs $x,z$ and $y,z$. Moreover, for all $\ell \in \{ \overline{xy}, \overline{xz}, \overline{yz} \}$ it holds that $ \{ x,y,z \} \not\subseteq \ell$.

    Let $t \in \overline{xy}\cap \overline{xz}\cap\overline{yz}$ and $u \in \{x,y,z\}$. If $\overrightarrow{tu}=\overrightarrow{ut}$ and $\overline{tu}\notin\{\overline{xz},\overline{xy},\overline{yz}\}$, then let $v \in Q \setminus \overline{tu}$. So, $t \notin \overline{uv}$ and the lines $\overline{xz},\overline{xy},\overline{yz},\overline{tu},\overline{uv}$ are pairwise distinct, a contradiction. Thus, for all $u\in \{x,y,z\}$ it holds that $\overrightarrow{tu} \neq \overrightarrow{ut}$.

    \begin{claim}
        For every $t \in \overline{xy} \cap \overline{xz} \cap \overline{yz}$ there is $u \in \{ x,y,z \}$ and $\ell \in \{ \overrightarrow{tu}, \overrightarrow{ut} \}$ such that $\{x,y,z\}\subset \ell$.    
    \end{claim}

    \begin{proof}
        Let $t \in \overline{xy} \cap \overline{xz} \cap \overline{yz}$ and $\{ u,v,w \} = \{ x,y,z \}$.
    
        First, consider the case $\{ \overrightarrow{tu}, \overrightarrow{ut} \} \subseteq \{ \overline{uv}, \overline{uw}\}$. Since $t \in \overline{uv} \cap \overline{uw}$, then $\{ \overrightarrow{tu}, \overrightarrow{ut} \} = \{\overline{uv},\overline{uw}\}$, since $v,w\in \overrightarrow{ut}\cup\overrightarrow{tu}$. If $\overrightarrow{tu} = \overline{uv}$ and $\overrightarrow{ut} = \overline{uw}$, then $v \notin \overrightarrow{ut}$. However, since $t \in \overline{uv}$, by Lemma \ref{l:nouniquetriples} we get that $vtu, tuv \in \betq$, so $u \in \overrightarrow{tv} \cap \overrightarrow{vt}$. Finally, since $t \in \overline{vw}$ then $w \in \overrightarrow{vt} \cup \overrightarrow{tv}$. So, either $\{ u,v,w \} \subset \overrightarrow{vt}$ or $\{ u,v,w \} \subset \overrightarrow{tv}$.  The case is analogous if $\overrightarrow{tu} = \overline{wu}$ and $\overrightarrow{ut} = \overline{uw}$.  

        Now, consider the case for which there is a line $\ell\in \{ \overrightarrow{tu}, \overrightarrow{ut} \} \setminus \{ \overline{xy}, \overline{xz}, \overline{yz} \}$. Without loss of generality, we may assume that $u=x$, otherwise rename $x,y$ and $z$. If $t$ and $x$ are symmetric, then $y,z \in \overline{tx}$ since $t \in \overline{xy} \cap \overline{xz}$ and we are done. So, assume that $\overrightarrow{xt} \neq \overrightarrow{tx}$.

        Suppose, for a contradiction, that $\lvert \overrightarrow{xt} \cap \{ x,y,z \} \rvert \leq 2$ and $\lvert \overrightarrow{tx} \cap \{ x,y,z \} \rvert \leq 2$. If $\overrightarrow{xt} \notin \{ \overline{xy}, \overline{xz}, \overline{yz} \}$, then $\linesq = \{ \overline{xy}, \overline{xz}, \overline{yz}, \overrightarrow{xt} \}$. Without loss of generality, we may assume that $y \notin \overrightarrow{xt}$. Otherwise, relabel $y$ and $z$. By Lemma \ref{l:covering}, we get $y \in \overrightarrow{tx}$ and Lemma \ref{l:nouniquetriples} implies that $ytx, txy \in \betq$. So, $x \in \overrightarrow{yt} \cap \overrightarrow{ty}$ which implies that $y$ and $t$ are symmetric and $\overline{yt} = \overline{xy}$. Since $z \notin \overline{xy} = \overline{yt}$, then $t \notin \overline{yz}$, a contradiction. Now, if $\overrightarrow{tx} \notin \{ \overline{xy}, \overline{xz}, \overline{yz} \}$ then $\linesq = \{ \overline{xy}, \overline{xz}, \overline{yz}, \overrightarrow{tx} \}$. Without loss of generality, we may assume that $y \notin \overrightarrow{tx}$. Otherwise, relabel $y$ and $z$. By Lemma \ref{l:covering}, we get $y \in \overrightarrow{xt}$ and Lemma \ref{l:nouniquetriples} implies that $yxt, xty \in \betq$. So, $x \in \overrightarrow{yt} \cap \overrightarrow{ty}$ which implies that $y$ and $t$ are symmetric and $\overline{yt} = \overline{xy}$. Since $z \notin \overline{xy} = \overline{yt}$, then $t \notin \overline{yz}$, a contradiction.
    \end{proof}
    
    By the previous claim, we have that $\linesq=\{\overline{xy}, \overline{xz}, \overline{yz}, \overrightarrow{tu} \}$. Let $s \notin \overrightarrow{tu}$. From Lemma \ref{l:covering} and Lemma \ref{l:uniquetriples} we have that $sut,uts \in \betq$. It is clear that for every $v \in \{ x,y,z \} \setminus \{ u \}$ it holds that $\overrightarrow{uv} \neq \overrightarrow{vu}$, otherwise we use the previous argument to arrive to a contradiction. Therefore, $\{\overrightarrow{tx}, \overrightarrow{xt}, \overrightarrow{ty}, \overrightarrow{yt}, \overrightarrow{tz}, \overrightarrow{zt} \} = \{ \overline{xy}, \overline{xz}, \overline{yz}, \overrightarrow{tu} \}$ and $\overrightarrow{tu} \in \{ \overrightarrow{tx}, \overrightarrow{xt} \} \cap \{ \overrightarrow{ty}, \overrightarrow{yt} \} \cap \{ \overrightarrow{tz}, \overrightarrow{zt} \}$. Since $s \notin \overrightarrow{tu}$, then $s \in \overline{xy} \cap \overline{xz} \cap \overline{yz}$. So, there is a $w \in \{ x,y,z \}$ such that the set $\{ \overrightarrow{sw}, \overrightarrow{ws} \} \setminus \{ \overline{xy}, \overline{xz}, \overline{yz} \}$ is non-empty, by our previous claim. This is a contradiction since $s \notin \overrightarrow{tu}$. 
\end{proof}

\begin{proposition}\label{p:nonthreesymmetric}
    Let $\qq$ be a quasimetric space of size $n \geq 5$ with no symmetric triples. Then, $\qq$ has a universal line or $\lvert \linesq \rvert \neq 4$.
\end{proposition}

\begin{proof}
    Suppose, for a contradiction, that $\qq$ has no symmetric triples and $\lvert \linesq \rvert = 4$, with no universal line. From Lemma \ref{l:nonsymmetric}, there is a symmetric pair of points $x$ and $y$. Since $\overline{xy}$ is not universal, let $z \in Q \setminus \overline{xy}$. By Lemma \ref{l:covering}, we also have that the sets $\{ x,z \}$ and $\{ y,z \}$ are symmetric such that $y \notin \overline{xz}$ and $x \notin \overline{yz}$. We conclude that $\linesq=\{\overline{xy},\overline{yz},\overline{xz}, \ell\}$ for some line $\ell \notin \{\overline{xy},\overline{yz},\overline{xz}\}$.

    If there is a point $w \in Q \setminus (\overline{xy} \cup \overline{xz})$, then $\overline{wx} = \ell$ and $ z, y \notin \ell$. This implies that   $\overline{wz}=\overline{wy}=\overline{yz}$, a contradiction. So, the set $Q \setminus (\overline{xy} \cup \overline{xz})$ is empty. The same argument shows that $Q \setminus (\overline{xy} \cup \overline{yz}) = Q \setminus (\overline{xz} \cup \overline{yz}) = \emptyset$. Therefore, we have that $Q = \overline{xy} \cup \overline{xz} = \overline{xy} \cup \overline{yz} = \overline{yz} \cup \overline{xz}$. From Lemma \ref{l:threeinternonempty}, it also follows that $\overline{xy} \cap \overline{yz} \cap \overline{xz} = \emptyset$, so
    \begin{equation}
        \label{eq:threesymlines}
        Q \setminus  \overline{xy} = \overline{xz} \cap \overline{yz}, \  Q\setminus \overline{xz} = \overline{xy} \cap \overline{yz} \textrm{ and } Q\setminus \overline{yz}= \overline{xz}\cap \overline{yx}. 
    \end{equation}
    Note that this holds for any symmetric pair of points $x$ and $y$ and a point $z \notin \overline{xy}$.

    Since $\qq$ has no universal lines, let $w \in Q \setminus \ell$. Without loss of generality, we may assume that $w \notin \overline{xy}$. So, $\overline{wx} = \overline{xz}$ and $\overline{wy} = \overline{yz}$. Since $\qq$ has no symmetric triples, then $\overrightarrow{wz} \neq \overrightarrow{zw}$ and $\{\overrightarrow{wz},\overrightarrow{zw}\}=\{\overline{wx},\overline{wy}\}$. Let $p,q \in Q$ such that $\ell = \overrightarrow{pq}$. If $\ell=\overline{pq}$, then $\{\overline{wp},\overline{wq}\}=\{\overline{xz},\overline{yz}\}$. So, by (\ref{eq:threesymlines}) applied to both $\{x,y,z\}$ and $\{w,p,q\}$, we get that $$Q\setminus \overline{xy}= \overline{xz}\cap \overline{yz}=\overline{wp}\cap \overline{wq}=Q\setminus \overline{pq},$$ which implies that $\ell=\overline{xy}$, a contradiction. Hence, the line $\ell$ is not symmetric so $\overrightarrow{pq} \neq \overrightarrow{qp}$. As $w \notin \ell=\overrightarrow{pq}$, by Lemma \ref{l:covering} we have that $w \in \overrightarrow{qp}$ and then $$ \overrightarrow{qp} \in \{ \overline{xz}, \overline{yz} \} = \{ \overline{xw}, \overline{yw} \}.$$ Without loss of generality, we may assume that $\overrightarrow{qp}=\overline{xz}$. Since $\overline{xy}\cap \overline{yz}\cap \overline{xz}$ is empty and $Q = \overline{xy} \cup \overline{yz}$ then $p \in \overline{xy}$ if and only if $p \notin \overline{yz}$ and $q \in \overline{xy}$ if and only if $q \notin \overline{yz}$.

    First, we consider the case $p \notin \overline{yz}$. Then, the sets $ \{ p, y \}$ and $ \{p, z \}$ are symmetric such that $\overline{pz}=\overline{xz}$ and $\overline{py}=\overline{xy}$. In this case, $\linesq = \{ \overrightarrow{pq}, \overline{py}, \overline{pz}, \overline{yz} \}$ and by relabeling $x$ and $p$ we get $$\linesq = \{ \overrightarrow{xq}, \overline{xy}, \overline{xz}, \overline{yz} \}.$$ So, we have that $\overrightarrow{qx} = \overline{xz}$ and $wqx,qxw,yxq,xqy \in \betq$ by Lemma \ref{l:nouniquetriples}. Hence, $q\in \overline{xy}$ which implies that $q\notin \overline{yz}$. Thus, the sets $\{ q, y \}$ and $\{ q, z \}$ are symmetric and $\overline{qy} \in \{\overrightarrow{xq}, \overline{xy} \}$ and $\overline{qz} \in \{ \overrightarrow{xq}, \overline{xz} \}$. Since $\ell=\overrightarrow{xq}$ is not symmetric then $\overline{qy} = \overline{xy}$ and $\overline{qz}=\overline{xz}$. Now, since $\qq$ has no symmetric triples, then $\overrightarrow{wz}\neq \overrightarrow{zw}$. So, 
    \begin{equation*}
        \{ \overrightarrow{wz}, \overrightarrow{zw} \} = \{ \overline{wx}, \overline{wy} \} = \{  \overline{xz}, \overline{yz} \} \text{ and } \overline{qz} = \overline{qw} = \overline{xz} = \overline{xw}, \ \overline{wy} = \overline{yz}.
    \end{equation*}
    If $( \overrightarrow{wz}, \overrightarrow{zw} ) = ( \overline{xw}, \overline{yw} ) = ( \overline{xz}, \overline{yz} )$, then $yzw,zwy,xwz,wzx,qwz,wzq \in \betq$. We conclude that $z \notin \overrightarrow{xq}$ as otherwise, we get that $w\in \overrightarrow{xq}$. In fact, as $wzq,xwz,qwz \in \betq$ then $ \{ zxq,xzq,xqz \} \cap \betq = \emptyset$. The case is analogous to show  that $z \notin \overrightarrow{xq}$ when $\overrightarrow{wz} = \overline{wy} = \overline{yz}$. In this case,  since $zwx,zwq,xzw\in \betq$ then $\{ zxq,xzq,xqz \} \cap \betq = \emptyset$. Note that the previous argument holds for any $u \notin \overline{xy}$. Indeed, since $\ell$ is not defined by symmetric points we have that $\overline{ux}=\overline{xz}$ and $\overline{uy} = \overline{zy}$. So, $\linesq = \{ \overrightarrow{xq}, \overline{xy}, \overline{yu}, \overline{xu} \}$. This shows that $\overrightarrow{xq} \subseteq \overline{xy}$.

    Now, let $u \notin \overrightarrow{xq}$. By Lemma \ref{l:covering}, it follows that $u \in \overrightarrow{qx}$ and by Lemma \ref{l:nouniquetriples} we get that 
    \begin{equation}
        \label{eq:pffourl1}
        \betq(\{u,x,q\}) = \{ qxu, uqx \}.
    \end{equation}
    If $u \in \overline{xy}$, then $u \notin \overline{yz}$ by Lemma \ref{l:threeinternonempty}. So, $\overline{uy} = \overline{xy}$ and $\overline{uz} = \overline{xz}$. If $x$ and $u$ are symmetric, then $\overline{xu} \in \{\overline{xy},\overline{xz}\}$ which implies that $x,z,u$ or $x,y,u$ is a symmetric triple. The case is analogous if $q$ and $u$ are symmetric. Therefore, $$\{\overrightarrow{ux},\overrightarrow{xu}\}=\{\overline{xy},\overline{xz}\}=\{\overrightarrow{uq},\overrightarrow{qu}\}.$$ Recall that $yxq \in \betq$. If $ \{ yux, yqu \} \cap \betq \neq \emptyset$, then $ \{ uxq, xqu \} \cap  \betq \neq \emptyset$, a contradiction since $u \notin \overrightarrow{xq}$. Thus, $yux, yqu \notin \betq$ which implies that $\overrightarrow{ux} = \overrightarrow{qu} = \overline{xz}$. Also, recall that $w \in \overline{xz} \setminus \overline{xy}$. Thus, $quw, uxw \in \betq$  which implies $qux\in \betq$, a contradiction with (\ref{eq:pffourl1}). So, $u \notin \overline{xy}$ and $\overline{xy} = \ell$, a contradiction.

    Finally, consider the case $p \in \overline{yz}$. Then, $p \notin \overline{xy}$ so the sets $\{ p, x \}$ and $\{ p, y \}$ are symmetric such that $\overline{px}=\overline{xz}$ and $\overline{py}=\overline{yz}$. In this case, $\linesq = \{ \overrightarrow{pq}, \overline{xy}, \overline{px}, \overline{py} \},$ and by relabelling $p$ and $z$ we get
    $$\linesq = \{ \overrightarrow{zq}, \overline{xy}, \overline{xz}, \overline{yz} \}.$$ 
    So, we have that $\overrightarrow{qz} = \overline{xz}$ and $wqz,qzw,yzq,zqy \in \betq$. Hence, $q \in \overline{zy}$ which implies that $q \notin \overline{xy}$. The sets $\{ q, x\}$ and $\{ q,y \}$ are symmetric and $\overline{qy}\in \{ \overrightarrow{zq}, \overline{yz} \}$ and $\overline{qx} \in \{ \overrightarrow{zq}, \overline{xz} \}$. Since $\ell = \overrightarrow{zq}$ is not symmetric then $\overline{qy} = \overline{yz}$ and $\overline{qx} = \overline{xz}$. Also, we have that the sets $\{ w, x \}$ and $\{ w,y \}$ are symmetric, so 
    \begin{equation*}
    \{ \overrightarrow{wq}, \overrightarrow{qw}\} = \{\overrightarrow{wz}, \overrightarrow{zw} \} = \{ \overline{xz}, \overline{yz} \} \text{ and } \overline{qx} = \overline{wx} = \overline{xz}, \  \overline{wy} = \overline{qy} = \overline{yz}.
    \end{equation*}
    If $\overrightarrow{qw}=\overline{xz}$, then $qwx\in \betq$ and then $zwx\in \betq$, since $qzw\in  \betq$. So, $\overrightarrow{zw}=\overline{xz}$ and $\overrightarrow{wz}=\overline{yz}.$ This implies that $ywz\in \betq$. As $wqz \in \betq$ we get that $yqz\in \betq$, a contradiction since $yzq \in \betq$. So, it must be that $\overrightarrow{qw} = \overline{yz}$ and $yqw \in \betq$. However, since $qzw \in \betq$ then $yqz \in \betq$, a contradiction since $y \notin \overrightarrow{qz} = \overline{xz}$.

\end{proof}

In the following sections, we present the characterizations of quasimetric spaces with 3 or 4 lines, none of which is universal. Before this, we establish some notation to present our results in these sections. For a labelled set $X = \{ x_1, \ldots, x_p \}$ of size $p$, we define the linear order $\ordx$ over $X$ where $x_i \ordx x_j$ if and only if $1 \leq i < j \leq p$.

\section{Quasimetric spaces with three lines}\label{s:three}

In this section, we classify quasimetric spaces with exactly three lines, none of which is universal. First, we build quasimetric spaces with three non-universal lines and $n$ points, for each $n\geq 4$. We then show that the betweennesses of all quasimetric spaces with three lines and no universal lines are isomorphic to one of those arising from our construction.

\begin{theorem}
\label{thm:existence}
    For every $n \geq 4$ and positive integers $p,q,r$ such that $p+q+r=n$, there is a quasimetric space $\qq$ of size $n$ satisfying the following:
    \begin{itemize}
        \item $\linesq=\{\ell_1,\ell_2,\ell_3\}$,
        \item $|\ell_1|=p+q$, $|\ell_2|=p+r$ and $|\ell_3|=q+r$.
    \end{itemize}
\end{theorem}

\begin{proof}
    Fix $n \geq 4$ and let $p,q,r$ be positive integers such that $p+q+r = n$. We will represent the quasimetric space as a strongly connected weighted digraph and define a quasimetric using the weights of the arcs. Let $(X=\{ x_1, \ldots, x_p \}, A_X)$ be a bidirected path of size $p$, $(Y=\{ y_1, \ldots, y_q \}, A_Y)$ a bidirected path of size $q$ and $(Z=\{ z_1, \ldots, z_r \}, A_Z)$ a bidirected path of size $r$. We define the weighted digraph $(Q,A,w)$ of size $n$ as
    \begin{align*}
        Q &= X \cup Y \cup Z, \\
        A &= A_X \cup A_Y \cup A_Z \cup \{ (x_1, z_1), (z_1, y_q), (y_q, x_1), (x_p, y_1), (y_1, z_r), (z_r, x_p) \},
    \end{align*}
    and the weight function $w:A \rightarrow [0, \infty)$ is defined as
    \begin{equation*}
        w(a) =
        \begin{cases}
            1 &\text{ if } a \in A_X \cup A_Y \cup A_Z,\\
            c &\text{ otherwise,}  
        \end{cases}
    \end{equation*}
    for an arbitrary constant $c > \max \{ p+q, p+r, q+r \}-2$. Note that this digraph is strongly connected, that is, for any two vertices in $u,v \in Q$, there is a directed path starting at $u$ and ending at $v$. For any such path, the length is defined as the sum of the weights of the arcs in the path. Thus, the weighted digraph $(Q,A,w)$ defines a quasimetric $\rho: Q \rightarrow [0, \infty)$ where $\rho(u,v)$ is the minimum length of a path from $u$ to $v$. It is clear that $\qq = (Q,\rho)$ is a quasimetric space of size $n$.

    To show that $\lvert \linesq \rvert = 3$, we will look at the maximal geodesics of $\qq$. One such maximal geodesic is $x_1 \ldots x_p y_1 \ldots y_q$. To see this, let $x_i \in X$ and $y_j, y_k \in Y$ such that $y_j \ordy y_k$. We have $\rho (x_i, y_k) = p-i + c + k-1$ and 
    \begin{equation}
    \label{eq:betxy}
        \rho (x_i, y_j) + \rho(y_j, y_k) = p-i + c + j-1 + j-k = p-1+c+k-1 = \rho(x_i, y_k),
    \end{equation}
    which implies that $x_i y_j y_k \in \betq$. The argument is analogous to show $x_i x_j y_k \in \betq$ for $x_i \ordy x_j$. Moreover, the geodesic $x_1 \ldots x_p y_1 \ldots y_q$ is maximal since for any $z \in Z$ it holds that $\min \{ \rho(x,z), \rho(z, x) \} \geq c$ and $\min \{ \rho(y, z), \rho(z, y) \} \geq c$, for all $x \in X$ and $y \in Y$. In a similar way, we can prove that $y_q \ldots y_1 z_r \ldots z_1$ and $x_1 \ldots x_p z_r \ldots z_1$ are also maximal geodesics. Thus, for any $x \in X$, $y \in Y$ and $z \in Z$ it holds that $\betq (\{ x,y,z \}) = \emptyset$ and so, these are the only maximal geodesics. Note that these geodesics have size $p+q$, $q+r$ and $ p+r$, respectively.

    Finally, we will show that each geodesic corresponds to a line in $\linesq$. For this, let $\ell_1 := \{ x_1, \ldots, x_p,$ $y_1, \ldots, y_q \}$. First, we show that $\overrightarrow{x y} = \overrightarrow{y x} = \ell_1$ for $x \in X$ and $y \in Y$. By (\ref{eq:betxy}), we have that $\ell_1 \subseteq \overrightarrow{x y} \;   (\text{or }\overrightarrow{y x})$. Moreover, if there is a $z \in Z \cap \overrightarrow{x y} \; (\text{or } \overrightarrow{y x})$ then there is a triple in $\betq$ that contains one point of $X,Y$ and $Z$ each, which is not possible. So, $\ell_1 = \overrightarrow{x y} = \overrightarrow{y_j x_i}$. Through a similar argument we have that $\overrightarrow{x_i x_j} = \ell_1$ for all $x_i \ordx x_j$ and $\overrightarrow{y_i y_j} = \ell_1$ for all $y_i \ordy y_j$.

    Note that we can apply the previous arguments to show that $\ell_2 := \{ y_q, \ldots, y_1, z_r, \ldots, z_1 \} = \overrightarrow{y z} = \overrightarrow{z y}$ for $y \in Y$ and $ z\in Z$ and $\overrightarrow{y_j y_i} = \ell_2 = \overrightarrow{z_j z_i}$ for $ y_i \ordy y_j$ and $ z_i \ordz z_j$. The reasoning is analogous to show that $\ell_3 := \{ x_1, \ldots, x_p, z_1, \ldots, z_r \} = \overrightarrow{x z} = \overrightarrow{z x}$ for $x \in X$ and $z \in Z$ and $\overrightarrow{x_j x_i} = \ell_2 = \overrightarrow{z_i z_j}$ for $x_i \ordx x_j$ and $z_i \ordz z_j$. It follows that $\lvert \ell_1 \rvert = p+q$, $\lvert \ell_2 \rvert = q+r$ and $\lvert \ell_3 \rvert = p+r$.
\end{proof}

\begin{figure}
    \centering
    \begin{subfigure}[t]{.49\textwidth}
        \centering
        \includegraphics[scale=.57]{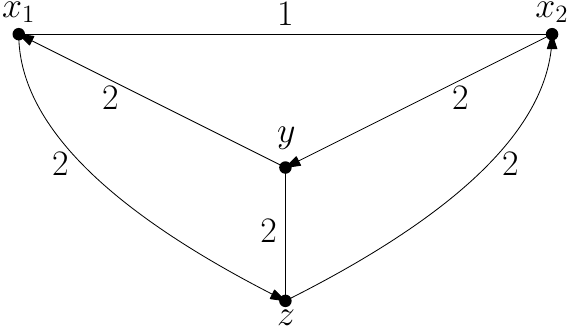}
        \caption{}
        \label{fig:small}
    \end{subfigure}
    \begin{subfigure}[t]{.49\textwidth}
        \centering
        \includegraphics[scale=.57]{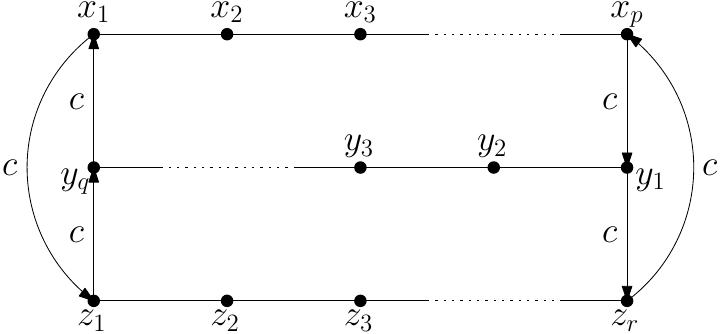}
        \caption{}
        \label{fig:three}
    \end{subfigure}
    \caption{ (a) A weighted strongly connected digraph representing a quasimetric space with four points and three lines, none of which are universal. The betweenness arising from this digraph is that presented in \cite{aramatzam} as the only betweenness arising from a quasimetric space that does not satisfy the Chen-Chv\'{a}tal Conjecture for quasimetric spaces. (b) The strongly connected weighted digraph representing the quasimetric space $\qq$ from Theorem \ref{thm:existence} where the undirected edges represented bidirected edges between two vertices and have weight 1. The remaining edges correspond to arcs of weight $c > \max \{ p+q, p+r, q+r \} - 2$. Theorem \ref{thm:uniquebet3} proves that any betweenness arising from a quasimetric space with exactly three lines, none of which are universal, is isomorphic to either $\betq_\qq$.}
    \label{fig:threelines}
\end{figure}

Let $p,q,r$ be positive integers. We denote by $\mathcal{C}^{p,q,r}$ the betweenness arising from the quasimetric space of size $n=p+q+r$ presented in the proof of Theorem \ref{thm:existence}. For simplicity, we will define certain sets,
\begin{equation}
\begin{alignedat}{1}
    \betq_V &= \{(v,v',v'')\in V^3\mid v\prec_V v'\prec_V v'' \lor v\succ_V v'\succ_V v''\},\\
    V^2_{\prec_V}W &= \{(v,v',w)\in V\times V\times W\vert v\prec_V v'\}, \\
    V^2_{\succ_V}W &= \{(v,v',w)\in V\times V\times W\vert v\succ_V v'\}, \\
    WV^2_{\prec_V} &= \{(w,v,v')\in W\times V\times V\vert v\prec_V v'\}, \\
    WV^2_{\succ_V} &= \{(w,v,v')\in W\times V\times V\vert v\succ_V v'\}, 
\end{alignedat}
\end{equation}
where $V,W\in \{X,Y,Z\}$.
Note that $\mathcal{C}^{p,q,r}$ is determined by the values $p,q$ and $r$ and so,

\begin{equation}
\label{eq:betthree}
\begin{alignedat}{4}
    \mathcal{C}^{p,q,r} = \; &\betq_X\ &&\cup\betq_Y &&\cup \betq_Z  &&\, \cup \\
    &X^2_{\prec_X}Y &&\cup YX^2_{\prec_X} &&\cup Y^2_{\prec_Y}X &&\cup XY^2_{\prec_Y} \, \cup \\
    &X^2_{\succ_X}Z &&\cup ZX^2_{\succ_X} &&\cup Z^2_{\prec_Z}X &&\cup XZ^2_{\prec_Z} \, \cup \\
    &Y^2_{\succ_Y}Z &&\cup ZY^2_{\succ_Y} &&\cup Z^2_{\succ_Z}Y &&\cup YZ^2_{\succ_Z},
\end{alignedat}
\end{equation}
where $|X|=p, |Y|=q$ and $|Z|=r$. Our next theorem states and proves that the betweennesses presented in (\ref{eq:betthree}) are the unique betweennesses, up to isomorphism, arising from quasimetric spaces with three lines, none of which is universal.

\begin{theorem}
\label{thm:uniquebet3}
    Let $\qq = (Q, \rho)$ be a quasimetric space of size $n\geq 4$ such that $\lvert \linesq \rvert = 3$ with no universal line. Then, there are integers $p,q,r$ such that $n=p+q+r$ and $\betq \cong \mathcal{C}^{p,q,r}$.
\end{theorem}

\begin{proof}
    Let $\ell \in \linesq$ such that for any $\ell' \in \linesq$, we have $\lvert \ell' \rvert \leq \lvert \ell \rvert$. Note that such a line always exists, but it may not be unique. Let $z$ be a point such that $z \notin \ell$. Such a point always exists since $\ell$ is not universal. Let $x \in \ell$ and consider the line $\ell_x := \overrightarrow{xz}$. Since $z \notin \ell$, then $\ell \neq \ell_x$. Moreover, since $\lvert \ell_x \rvert \leq \lvert \ell \rvert$, then there is a point $y \in \ell \setminus \ell_x$. This implies that the line $\ell_y := \overrightarrow{yz} \notin \{ \ell, \ell_x \}$. It follows that $\linesq = \{ \ell, \ell_x, \ell_y \}$. Additionally, we have that $\overrightarrow{yz} = \overrightarrow{zy}$ since $y \notin \ell_x$ and $z \notin \ell$.

    \begin{claim}
    \label{clm:xzzx}
        The points $x$ and $z$ are symmetric.
    \end{claim}

    \begin{proof}
        Suppose, for a contradiction, that $\overrightarrow{zx} \in \{ \ell_y, \ell \}$. Since $z \notin \ell$, then $\overrightarrow{zx} = \ell_y$. This implies that $x \in \overline{yz}$. By Lemma \ref{l:nouniquetriples}, since $y \notin \ell_x$ then $zxy, yzx \in \betq$ since $y \notin \ell_x$. So, it follows that $\overrightarrow{xy} = \overrightarrow{yx} = \ell_y$, as $z \notin \ell$.

        Now, since $\lvert \ell_y \rvert \leq \lvert \ell \rvert$, then there is a point $w \in \ell \setminus \ell_y$. This yields that $\ell_x = \overrightarrow{wz} = \overrightarrow{zw}$. Now, as we had that $\overrightarrow{zx} = \ell_y$, then it must be that $wxz, xzw \in \betq$, by Lemma \ref{l:nouniquetriples}. So, $\overrightarrow{wx}= \overrightarrow{xw}= \ell_x$. Finally, the line $\overrightarrow{wy}$ must be equal to $\ell$ since $w \notin \ell_y$. So, $x \in \overrightarrow{wy}$ and so $y \in \overrightarrow{wx} \cup \overrightarrow{xw} = \ell_x$, a contradiction.
    \end{proof}

    As a consequence of Claim \ref{clm:xzzx}, we get that $\ell = \overrightarrow{xy} = \overrightarrow{yx}$. So, all lines are symmetric. By Lemma \ref{l:threeinternonempty}, the set $\ell \cap \ell_x \cap \ell_y$ is empty. Let $X = \ell \cap \ell_x$, $Y = \ell \cap \ell_y$ and $Z= \ell_x \cap \ell_y$. Note that none of these sets is empty since $x \in X$, $y \in Y$ and $z \in Z$. It is clear that $\{ X,Y,Z \}$ is a partition of $Q$ since every point must be in one line and $\ell \cap \ell_x \cap \ell_y=\emptyset$. Finally, let $p = \lvert X \rvert, q = \lvert Y \rvert$ and $r = \lvert Z \rvert$.

    \begin{claim}
        \label{clm:xorder}
        For every $W \in \{ X,Y,Z \}$ with $\lvert W \rvert \geq 2$ there is a linear order $\ordw$ over W such that for all $w \ordw w' \ordw w^*$ it holds that $\overrightarrow{w w'} = \overrightarrow{w w^*} = \overrightarrow{w' w^*} \neq \overrightarrow{w' w} = \overrightarrow{w^* w} = \overrightarrow{w^* w}$.
    \end{claim}

    \begin{proof}

        It suffices to show the statement for $W = X$ since the proof is analogous for the remaining cases. So, assume that $ \lvert X \rvert \geq 2$ and let $x, x' \in X$, $y \in Y$ and $z \in Z$. Since the set $\ell \cap \ell_x \cap \ell_y$ is empty, then the sets $\{ x,y \}$ and $\{ x',y \}$ are symmetric and $\overline{xy} = \ell = \overline{x'y}$. So, $y \in \overrightarrow{xx'} \cup \overrightarrow{x'x}$. The same reasoning applies to $x,x'$ and $z$ so, we have that $z \in \overrightarrow{xx'} \cup \overrightarrow{x'x}$. So, we conclude that $\lvert \{y,z\} \cap \overrightarrow{xx'} \rvert = \lvert \{y,z\} \cap \overrightarrow{x'x} \rvert = 1$.

        Now, consider the digraph $D = (X,A)$ where $(x,x') \in A$ if and only if $y \in \overrightarrow{xx'}$. Since ${y \in \overrightarrow{xx'} \vartriangle \overrightarrow{x'x}}$, the digraph $D$ is a tournament. To prove the claim, it suffices to show that $D$ is a transitive tournament since this would imply the existence of such a linear order $\ordx$ over $X$. So, let $(x,x'), (x',x^*) \in A$. Since $y \notin \overrightarrow{x'x}$ then $ xx'y, yxx' \in \betq$, by Lemma \ref{l:nouniquetriples}. By the same arguments, it follows that $x'x^* y, y x' x^* \in \betq$. So, applying (\ref{prop:2}) yields $yx x^* , x x^* y, xx'x^* \in \betq$. The case is analogous to show that $zx^*x, x^* x z, x^*x'x \in \betq$.
    \end{proof}

    By Claim \ref{clm:xorder}, we get
    \begin{alignat*}{3}
        X &= \{ x_1, \ldots, x_p \} &&\text{ such that } \ell = \overrightarrow{x_i x_j} \neq \overrightarrow{x_j x_i} = \ell_x &&\text{ for } x_i \ordx x_j, \\
        Y &= \{ y_1, \ldots, y_q \} &&\text{ such that } \ell_y = \overrightarrow{y_i y_j} \neq \overrightarrow{y_j y_i} = \ell &&\text{ for } y_i \ordy y_j, \\
        Z &= \{ z_1, \ldots, z_r \} &&\text{ such that } \ell_x = \overrightarrow{z_i z_j} \neq \overrightarrow{z_j z_i} = \ell_y &&\text{ for } z_i \ordz z_j,
    \end{alignat*}
    and also
    \begin{alignat*}{2}
        &\ell &&= \{x_1, \ldots, x_p, y_1, \ldots, y_q \}, \\
        &\ell_x &&= \{ x_1, \ldots, x_p, z_1, \ldots, z_r \}, \\
        &\ell_y &&= \{ y_1, \ldots, y_q, z_1, \ldots, z_r \}.
    \end{alignat*}
    It is clear that each line corresponds to a maximal geodesic. Namely, the line $\ell$ corresponds to $x_1 \ldots x_p y_1 \ldots y_q$, $\ell_x$ to $x_q \ldots x_1 z_1 \ldots z_r$ and $\ell_y$ to $y_q \ldots y_1 z_r \ldots z_1$. Note that these are the only maximal geodesics, since $\ell \cap \ell_x \cap \ell_y$ is empty. Thus, we have that $\betq \cong \mathcal{C}^{p,q,r}$.
\end{proof}
Let $p_3(n)$ denote the number of partitions of the integer $n$ into three parts.

\begin{corollary}
For all $n\geq 4$ there are $p_3(n)$ non isomorphic betweennesses arising from quasimetric spaces on $n$ points without universal lines and with exactly three lines.
\end{corollary}
\section{Quasimetric spaces with four lines}
\label{s:four}

In this section, we characterize quasimetric spaces with four lines, none of which is universal. Throughout the section, let $\qq=(Q, \rho)$ be a quasimetric space of size $n \geq 4$ such that $\lvert \linesq \rvert = 4$, and no line is universal. We now present the analogous version of Theorem \ref{thm:existence} for this section. 

\begin{theorem}
    \label{thm:exisfour}
    For every $n \geq 4$ and positive integers $p,q,r$ such that $p+q+r+1=n$, there are two quasimetric spaces $\mathcal{Q}_1$ and $\mathcal{Q}_2$ of size $n$ whose betweennesses are not isomorphic and satisfying the following, for each $i=1,2$:
    \begin{itemize}
        \item $\mathcal{L}_{\mathcal{Q}_i}=\{\ell_1,\ell_2,\ell_3,\ell_4\}$,
        \item $|\ell_1|=n-1$, $|\ell_2|=p+1$,  $|\ell_3|=q+1$  and $|\ell_4|=r+1$. 
    \end{itemize} 
    \end{theorem}

\begin{proof}
     Fix $n \geq 4$ and let $p,q,r$ be positive integers such that $p+q+r+1= n$. We will represent the quasimetric spaces as strongly connected weighted digraphs $D_1$ and $D_2$ and define the distance between $u$ and $v$ as the minimum length of a path from $u$ to $v$. Let $(X=\{ x_1, \ldots, x_p \}, A_X)$ be a bidirected path of size $p$, $(Y=\{ y_1, \ldots, y_q \}, A_Y)$ a bidirected path of size $q$, $(Z=\{ z_1, \ldots, z_r \}, A_Z)$ a bidirected path of size $r$ and $u$ a point such that $u \notin X \cup Y \cup Z$. For each $i=1,2$, we define the weighted digraph $D_i=(Q,A_i,w_i)$ of size $n$ as
    \begin{align*}
        Q = \ &X \cup Y \cup Z \cup \{ u \}, \\
        A_1 = \ &A_X \cup A_Y \cup A_Z \cup \\
        &\{ (x_1, u), (u, x_p), (y_1, u),(u, y_q), (z_1, u), (u, z_r), (x_p, y_1), (y_q, z_1), (z_r, x_1) \}\\
        A_2 = \ &A_X \cup A_Y \cup A_Z \cup \\
        &\{ (x_1, u), (u, x_p), (y_1, u),(u, y_q), (z_1, u), (u, z_r), (x_p, y_1), (y_q, z_1), (z_r, y_1), (y_q,x_1) \}.
    \end{align*}
    The weight function $w_1:A_1 \rightarrow [0, \infty)$ is defined as 
    \begin{equation*}
        w_1(a) =
        \begin{cases}
            1 &\text{ if } a \in A_X \cup A_Y \cup A_Z,\\
            c &\text{ if } a \in \{  (x_p, y_1), (y_q, z_1), (z_r, x_1) \}, \\
            2c &\text{ otherwise,} 
        \end{cases}
    \end{equation*}
    for an arbitrary constant $c$ such that $2c > p+q+r-3$. The weight function $w_2:A_2 \rightarrow [0, \infty)$ is defined as 
    \begin{equation*}
        w_2(a) =
        \begin{cases}
            1 &\text{ if } a \in A_X \cup A_Y \cup A_Z,\\
            p &\text{ if } a=(x_p, y_1) \\
            r &\text{ if } a=(z_r, y_1) \\
            q+p &\text{ if } a=(y_q,x_1) \\
            q+r &\text{ if } a=(y_q, z_1) \\
            p+r &\text{ if } a=(u, y_q)\\
            c &\text{ otherwise,} 
        \end{cases}
    \end{equation*}
    where $c=p+q+r-2$. Let $\betq_1$ and $\betq_2$ be the corresponding betweennesses defined by the quasimetric spaces $\qq_1$ and $\qq_2$ associated to $D_1$ and $D_2$, respectively.
    
    To show that $\lvert \mathcal{L}_{\mathcal{Q}_i} \rvert = 4$, we will focus on the maximal geodesics of $\qq_i$ for $i \in \{ 1,2 \}$ (such as we did in the proof of Theorem \ref{thm:existence}). First, let $i=1$. It is clear that one such maximal geodesic of $\qq_1$ is $x_1 \ldots x_p y_1 \ldots y_q z_1 \ldots z_r$. In this case, the sequences $z_1 \ldots z_r x_1 \ldots x_p y_1 \ldots y_q$ and $y_1 \ldots y_q z_1 \ldots z_r x_1 \ldots x_p$ are also maximal geodesics. Also, for every $V \in \{ X,Y,Z \}$, the set $V \cup \{ u \}$ is geodetic. The following claim presents the respective geodesics.
    \begin{claim}
    \label{clm:geodesics}
        For any $V \in \{ X,Y,Z \}$ and $s= \lvert V \rvert$, the sequences $v_s\ldots v_1 z$ and $z v_s \ldots v_1$ are maximal geodesics.
    \end{claim}

    \begin{proof}
        For any $V \in \{ X,Y,Z \}$, the sequence $v_s \ldots v_1$ is a geodesic. So, it is left to show that $v_j v_i u, u v_j v_i \in \betq_1$ for $ v_i \ordv v_j$. It suffices to show the statement for $V=X$ since the arguments are analogous for the remaining cases. By definition of the weight function $w_1$, for $x_i \ordx x_j$, it holds that
        \begin{align*}
            \rho (x_j , x_i) + \rho(x_i, u) &= j - i + i - 1 + 2c = 2c + j - 1 = \rho(x_j, u) \text{ and }\\
            \rho (u, x_j) + \rho(x_j, x_i) &= 2c + p - j + j - i = 2c + p - i = \rho(u, x_i).
        \end{align*}
        So, the sequences $x_p\ldots x_1 u$ and $u x_p \ldots x_1$ are geodesics. Finally, maximality of these geodesics follows from the fact that $\max \{ \rho(u,v), \rho(v,u) \} \geq 2c$ for any $v \in X \cup Y \cup Z$.
    \end{proof}

    By Claim \ref{clm:geodesics}, the sets
    \begin{align*}
        \ell_1 &:= \{ x_1, \ldots, x_p, y_1, \ldots, y_q, z_1, \ldots, z_r \}, \\
        \ell_2 &:= \{ x_1, \ldots, x_p, u \}, \\
        \ell_3 &:= \{ y_1, \ldots, y_q, u \}, \\
        \ell_4 &:= \{ z_1, \ldots, z_r, u \},
    \end{align*}
    are geodetic and correspond to the maximal geodesics of $\qq_1$. It remains to show that ${\linesq}_1 = \{ \ell_1, \ell_2, \ell_3, \ell_4 \}$. However, by Claim \ref{clm:geodesics} and the fact that $x_1 \ldots x_p y_1 \ldots y_q z_1 \ldots z_r$ is a maximal geodesic, we have that
    \begin{equation*}
    \overrightarrow{ab} = 
    \begin{cases}
        \ell_1 \text{ if } a \ordv b \text { for some } V \in \{ X,Y,Z \} \text{ or there is no } V \in \{X,Y,Z \} \text{ such that } a,b \in V, \\
        \ell_2 \text{ if } b \ordx a \text{ or } u \in \{ a,b \}, \\
        \ell_3 \text{ if } b \ordy a \text{ or } u \in \{ a,b \}, \\
        \ell_4 \text{ if } b \ordz a \text{ or } u \in \{ a,b \}.
    \end{cases}
    \end{equation*}
    This concludes the proof for $\qq_1$.

    Now, let $i=2$. So, the sequence $z_1 \ldots z_r y_1 \ldots y_q x_1 \ldots x_p$ is a  maximal geodesic. Also, for every $V \in \{ X,Y,Z \}$, the set $V \cup \{ u \}$ is geodetic. The following claim presents the respective geodesics.
    \begin{claim}
    \label{clm:geodesics2}
        For any $V \in \{ X,Y,Z \}$ and $s= \lvert V \rvert$, the sequences $v_s\ldots v_1 z$ and $z v_s \ldots v_1$ are maximal geodesics.
    \end{claim}

    \begin{proof}
        For any $V \in \{ X,Y,Z \}$, the sequence $v_s \ldots v_1$ is a geodesic. So, it is left to show that $v_j v_i u, u v_j v_i \in \betq_2$ for $ v_i \ordv v_j$. It suffices to show the statement for $V=X$ and $V=Y$, since for $V=Z$,  the argument is analogous to that for $V=X$. By definition of the weight function $w_2$, for $x_i \ordx x_j$, it holds that
        \begin{align*}
            \rho (x_j , x_i) + \rho(x_i, u) &= j - i + i - 1 + c = c + j - 1 = \rho(x_j, u) \text{ and }\\
            \rho (u, x_j) + \rho(x_j, x_i) &= c + p - j + j - i = c + p - i = \rho(u, x_i).
        \end{align*}
        So, the sequences $x_p\ldots x_1 u$ and $u x_p \ldots x_1$ are geodesics. Finally, maximality of these geodesics follows from the fact that $\rho(v,u)+\rho(u,w)\geq 2c>\rho(v,w)$ for any $v,w \in X \cup Z$ and $\rho(v,u)+\rho(u,w)\geq c+p+r>\rho(v,w)$ for any $v\in X\cup Z, w\in Y$. 
        
        By definition of the weight function $w_2$, for $y_i \ordx y_j$, it holds that
        \begin{align*}
            \rho (y_j , y_i) + \rho(y_i, u) &= j - i + i - 1 + c = c + j - 1 = \rho(y_j, u) \text{ and }\\
            \rho (u, y_j) + \rho(y_j, y_i) &= p + r + q - j + j - i = c - i = \rho(u, y_i).
        \end{align*}
        So, the sequences $y_q\ldots y_1 u$ and $u y_q \ldots y_1$ are geodesics. Finally, maximality of these geodesics follows from the fact that $\rho(v,u)+\rho(u,w)\geq c+p+r>\rho(v,w)$ for any $v \in X \cup Z$, $w\in Y$, and $\rho(v,u)+\rho(u,w)\geq 2c>\rho(v,w)$ for any $w\in X\cup Z, v\in Y$.
    \end{proof}

    By Claim \ref{clm:geodesics2}, the sets
    \begin{align*}
        \ell_1 &:= \{ x_1, \ldots, x_p, y_1, \ldots, y_q, z_1, \ldots, z_r \}, \\
        \ell_2 &:= \{ x_1, \ldots, x_p, u \}, \\
        \ell_3 &:= \{ y_1, \ldots, y_q, u \}, \\
        \ell_4 &:= \{ z_1, \ldots, z_r, u \},
    \end{align*}
    are geodetic and correspond to the maximal geodesics of $\qq_2$. It remains to show that $\mathcal{L}_{\qq_2} = \{ \ell_1, \ell_2, \ell_3, \ell_4 \}$. However, by Claim \ref{clm:geodesics2} and the fact that $x_1 \ldots x_p y_1 \ldots y_q z_1 \ldots z_r$ is a maximal geodesic, we have that $z_1 \ldots z_r y_1 \ldots y_q x_1 \ldots x_p$ is also a maximal geodesic.
    \begin{equation*}
    \overrightarrow{ab} = 
    \begin{cases}
        \ell_1 \text{ if } a \ordv b \text { for some } V \in \{ X,Y,Z \} \text{ or there is no } V \in \{X,Y,Z \} \text{ such that } a,b \in V, \\
        \ell_2 \text{ if } b \ordx a \text{ or } u \in \{ a,b \}, \\
        \ell_3 \text{ if } b \ordy a \text{ or } u \in \{ a,b \}, \\
        \ell_4 \text{ if } b \ordz a \text{ or } u \in \{ a,b \}.
    \end{cases}
    \end{equation*}
    This concludes the proof for $\qq_2$. Finally, note that two isomorphic betweennesses must have the same cardinality.  Since $\lvert \betq_1 \rvert = \lvert \betq_2 \rvert + pqr$, we conclude that $\betq_1$ and $\betq_2$ are not isomorphic.
\end{proof}

\begin{figure}
    \centering
    \begin{minipage}[t]{.49\textwidth}
        \centering
        \includegraphics[scale=.57]{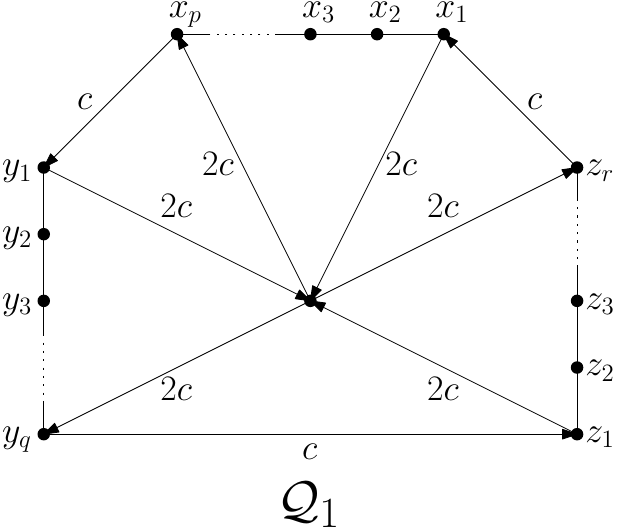}
        \label{fig:one}
    \end{minipage}
    \begin{minipage}[t]{.49\textwidth}
        \centering
        \includegraphics[scale=.57]{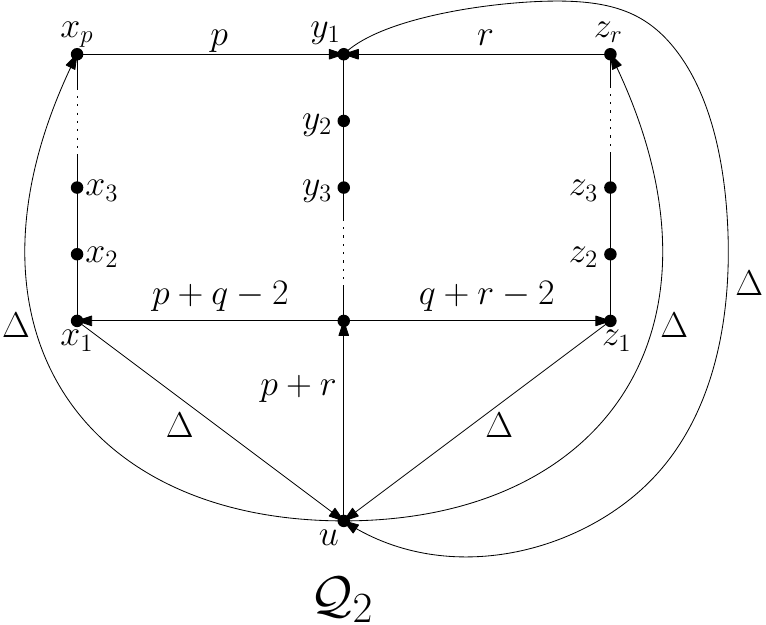}
        \label{fig:two}
    \end{minipage}
    \caption{The strongly connected weighted digraph representing the quasimetric spaces $\qq_1$ and $\qq_2$ from Theorem \ref{thm:exisfour} where the undirected edges represented bidirected edges between two vertices and have weight 1. The remaining edges correspond to arcs of weight $c > (p+q+r-3)/2$ and $\Delta = p+q+r-2$. Theorem \ref{t:fourthreesymmetric} proves that any betweenness arising from a quasimetric space with exactly four lines, none of which are universal, is isomorphic to either $\mathcal{D}^{p,q,r}_1$ or $\mathcal{D}^{p,q,r}_2$.}
    \label{fig:enter-label}
\end{figure}

\begin{proposition}
\label{prop:threesym}
    Let $\qq$ be a quasimetric space with four lines, none of which is universal, and let $\ell \in \linesq$ be a symmetric line such that $\ell = \overline{xy} = \overline{yz} = \overline{xz}$, for three distinct points $x$, $y$ and $z$. Then, $\lvert \ell \rvert = n-1$.
\end{proposition}

\begin{proof}
    Suppose, for a contradiction, that $\lvert \ell \rvert \leq n-2$ and let $u,v \in V \setminus \ell$. Let $a,b \in \{ x,y,z \}$. Since $v \notin \ell = \overline{ab}$, then $b \notin \overrightarrow{av} \cup \overrightarrow{va}$. So, any point $a \in \{ x,y,z \}$ is symmetric with $v$ and, so $\linesq = \{ \ell, \overline{xv}, \overline{yv}, \overline{zv} \}$. The same reasoning can be applied to $u$, which yields that $\overline{xv} = \overline{xu}$, $\overline{yv} = \overline{yu}$ and $\overline{zv} = \overline{zu}$. So, for all $a \in \{ x,y,z \}$, we have that $a \in \overline{uv} \cup \overline{vu}$. By pigeonhole, one of the lines $\overline{uv}, \overline{vu}$ contains two elements of $\{ x,y,z \}$. However, this is a contradiction since the only lines in $\linesq$ containing $u$ and $v$ contain at most one element of $\{ x,y,z \}$.
\end{proof}

Let $p,q,r$ be positive integers. For $i \in \{ 1 ,2 \}$, we denote by $\mathcal{D}_i^{p,q,r}$ the betweenness arising from the quasimetric space of size $n=p+q+r+1$ presented in the proof of Theorem \ref{thm:exisfour}. Note that $\mathcal{D}_i^{p,q,r}$ is determined by $p,q$ and $r$. So,
\begin{equation*}
\begin{alignedat}{6}
    \mathcal{D}_1^{p,q,r} = \; &\betq_X&&\cup\betq_Y&&\cup \betq_Z&&\cup  XYZ && \cup YZX &&\cup ZXY  \\
    &\cup X^2_{\prec_X}Y &&\cup X^2_{\prec_X}Z &&\cup Y^2_{\prec_Y}X &&\cup Y^2_{\prec_Y}Z&&\cup Z^2_{\prec_Z}X &&\cup Z^2_{\prec_Z}Y  \\ 
    &\cup YX^2_{\prec_X} &&\cup ZX^2_{\prec_X} &&\cup XY^2_{\prec_Y} &&\cup ZY^2_{\prec_Y}&&\cup XZ^2_{\prec_Z} &&\cup YZ^2_{\prec_Z}  \\ 
     &\cup uX^2_{\succ_X} &&\cup X^2_{\succ_X} u &&\cup u Y^2_{\succ_Y} && \cup Y^2_{\succ_Y}u &&\cup uZ^2_{\succ_Z}&&\cup Z^2_{\succ_Z}u ,  \\
    \mathcal{D}_2^{p,q,r} = \; &\betq_X&&\cup\betq_Y&&\cup \betq_Z&&\cup  XYZ && \cup ZYX  \\
    &\cup X^2_{\prec_X}Y &&\cup X^2_{\prec_X}Z &&\cup Y^2_{\prec_Y}X &&\cup Y^2_{\prec_Y}Z&&\cup Z^2_{\prec_Z}X &&\cup Z^2_{\prec_Z}Y  \\ 
    &\cup YX^2_{\prec_X} &&\cup ZX^2_{\prec_X} &&\cup XY^2_{\prec_Y} &&\cup ZY^2_{\prec_Y}&&\cup XZ^2_{\prec_Z} &&\cup YZ^2_{\prec_Z}  \\ 
     &\cup uX^2_{\succ_X} &&\cup X^2_{\succ_X} u &&\cup u Y^2_{\succ_Y} && \cup Y^2_{\succ_Y}u &&\cup uZ^2_{\succ_Z}&&\cup Z^2_{\succ_Z}u ,  
\end{alignedat}
\end{equation*}
where $uV^2_{\succ_V}=\{(u,v,v')\in \{u\}\times V\times V\vert v\succ v'\}$ and $V^2_{\succ_V} u=\{(v,v',u)\in V\times V\times \{u\}\vert v\succ v'\}$ with $V\in \{X,Y,Z\}$, and $UVW=U\times V\times W$, with $\{U,W,V\}=\{X,Y,Z\}$.
\begin{theorem}
\label{t:fourthreesymmetric}
    Let $\qq = (Q,\rho)$ be a quasimetric space of size $n$ such that $\lvert \linesq \rvert = 4$ with no universal line. Then, there are integers $p,q,r$ such that $n=p+q+r+1$ and $\betq \cong \mathcal{D}_1^{p,q,r}$ or $\betq \cong \mathcal{D}_2^{p,q,r}$. 
\end{theorem}

\begin{proof} 
    From Proposition \ref{p:nonthreesymmetric} we know that there are $x,y,z \in Q$ and $\ell \in \linesq$ such that $\ell = \overline{xy} = \overline{yz} = \overline{xz}$. It follows from Proposition \ref{prop:threesym} that $\lvert \ell \rvert = n-1$, so let $u$ be the unique point in $Q \setminus \ell$. For any two distinct $a,b \in \{ x,y,z \}$, we have $b \notin \overrightarrow{ua} \cup \overrightarrow{au}$ which implies that $\{ a,b,u \}$ is a symmetric set and $\overline{ub} \neq \overline{ua}$. So, for $a \in \{ x,y,z \}$, write $\ell_a = \overline{ua}$ and $\linesq = \{ \ell, \ell_x, \ell_y, \ell_z \}$. By Lemma \ref{l:covering}, for any distinct $a,b \in \{x,y,z \}$, the set $\ell \cap \ell_a \cap \ell_b$ is empty.
    
    Now, let $X = \ell \cap \ell_x$, $Y = \ell \cap \ell_y$ and $Z = \ell \cap \ell_z$ and note that these sets are nonempty. The following claim is analogous to Claim \ref{clm:xorder}.

    \begin{claim}
    \label{clm:orderallfour}
        Let $V \in \{ X,Y,Z \}$ (resp. $v \in \{ x,y,z \}$). If $ \lvert V \rvert \geq 2$, then there is a linear order $\ordv$ over $V$ such that for all $ v' \ordv v^*$ it holds that $\overrightarrow{v' v^*} = \ell$ and $\overrightarrow{v^* v'} = \ell_v$.
    \end{claim}

    \begin{proof}
        It suffices to show that the statement holds for $V=X$ (resp. $v=x$), as the arguments are analogous for the remaining cases. So, assume $\lvert X \rvert \geq 2$ and let $x' \in X \setminus \{x\}$. Since $x' \in \ell_x$, then $u \in \overrightarrow{xx'} \cup \overrightarrow{x'x}$. Also, since $x' \in \ell$ then $y,z \in \overrightarrow{xx'} \cup \overrightarrow{x'x}$. Given that there is no line containing $x,y$ and $u$ (or $x,z$ and $u$), we must have that $\{ \overrightarrow{xx'}, \overrightarrow{x'x} \} = \{ \ell, \ell_x \}$ and $X = \overrightarrow{xx'} \cap \overrightarrow{x'x}$. So, 
        \begin{equation}
        \label{propty:xx'}
            u \in \overrightarrow{xx'} \Leftrightarrow y,z \in \overrightarrow{x'x}.
        \end{equation}

        By (\ref{propty:xx'}), we have $u \in \overrightarrow{xx'}$ if and only if $u \notin \overrightarrow{x'x}$. So, if $u \in \overrightarrow{xx'}$, then $uxx', xx'u \in \betq$. On the other hand, if $u \in \overrightarrow{x'x}$ then $ux'x, x'xu \in \betq$. Now, for any $x^* \in X \setminus \{ x,x' \}$ note that $\{ \overrightarrow{xx^*}, \overrightarrow{x^*x} \} = \{ \overrightarrow{x'x^*}, \overrightarrow{x^*x'} \} = \{ \overrightarrow{xx'}, \overrightarrow{x'x} \} = \{ \ell, \ell_x \}$. Analogous to (\ref{propty:xx'}), we have
        \begin{equation}
        \label{propty:x*x'}
            u \in \overrightarrow{x'x^*} \Leftrightarrow y,z \in \overrightarrow{x^*x'}.
        \end{equation}
        By (\ref{propty:x*x'}), we have that $u \in \overrightarrow{x'x^*}$ if and only if $u \notin \overrightarrow{x^*x'}$. In the same way as before, either $ux'x^*, x'x^*u \in \betq$ or $ux^*x', x^*x'u \in \betq$ hold.

        Let $D = (X,A)$ be a digraph where $(x',x^*) \in A$ if and only if $u \in \overrightarrow{x'x^*}$. Since $u \in \overrightarrow{x'x^*} \vartriangle \overrightarrow{x^*x'}$, the digraph $D$ is a tournament. To prove our claim, it suffices to show that $D$ is a transitive tournament since this would imply the existence of such an order in $X$. So, let $(x', x^*), (x^*,x^+) \in A$. The fact that $u \in \overrightarrow{x' x^*} \setminus \overrightarrow{x' x^*}$ implies $u x^* x', x^* u x', x^*x' u, x' u x^* \notin \betq$ and $u x' x^*, x' x^*u \in \betq$. In the same way, it follows that $u x^* x^+, x^* x^+ u \in \betq$. Using (\ref{prop:2}) gives $u x' x^+ , x' x^* x^+ \in \betq$, so $(x', x^+) \in A$. The reverse linear order of the vertices of $D$ satisfies our claim.
    \end{proof}

    For the rest of the proof, we write $X = \{x_1, \ldots, x_p \}$, $Y= \{y_1, \ldots, y_q \}$ and $Z = \{z_1, \ldots, z_r \}$ for $p = \lvert X \rvert$, $ q=\lvert Y \rvert$ and $r = \lvert Z \rvert$, with respect to the order obtained from Claim \ref{clm:orderallfour}.

    \begin{claim}
    \label{clm:largegeofour}
        Let $V,W \in \{ X,Y,Z \}$ with $\lvert V \rvert = s$ and $\lvert W \rvert =t$ such that $\{ s,t \} \neq \{ 1 \}$. Then, for $v_i, v_j \in V$ with $v_i \ordv v_j$ and $w \in W$ it holds that $v_i v_j w, w v_i v_j \in \betq$. Moreover, the points $v$ and $w$ are symmetric, for any $v \in V$ and $w \in W$.
    \end{claim}

    \begin{proof}    
        It follows from Claim \ref{clm:orderallfour} that $w \in \overrightarrow{v_i v_j} = \ell$ and $w \notin \overrightarrow{v_j v_i} = \ell_x$, so $w v_j v_i, v_j w v_i, v_j v_i w \notin \betq$. Since $\lvert \linesq \rvert = 4$, then by Lemma \ref{l:nouniquetriples}, it must be that $v_i v_j w, w v_i v_j \in \betq$ and $v_i w v_j \notin \betq$. Now, for $v \in V$ and $w \in W$ the lines $\overrightarrow{vw}$ and $\overrightarrow{wv}$ satisfy that $V \cup W \subset \overrightarrow{vw} \cap \overrightarrow{wv}$. Since $\ell$ is the only line satisfying $V \cup W \subset \ell$, then $v$ and $w$ are symmetric and $\ell = \overline{vw}$.
    \end{proof}

    By Claim \ref{clm:largegeofour}, exactly one of the sequences $x_1 \ldots x_p y_1 \ldots y_q z_1 \ldots z_r$ or $y_1 \ldots y_q x_1 \ldots x_p z_1 \ldots z_r$ is a maximal geodesic. Without loss of generality, we may assume that $x_1 \ldots x_p y_1 \ldots y_q z_1 \ldots z_r$ is a maximal geodesic, otherwise relabel the points of $X$ and $Y$ accordingly. 

    First, we consider the case when $\rho (z_r,x_1) < \rho (z_r,y_1)$. Then, the sequences $y_1 \ldots y_q z_1 \ldots z_r x_1 \ldots x_p$ and $z_1 \ldots z_r x_1 \ldots x_p y_1 \ldots y_q$ are also maximal geodesics. Moreover, these are the only maximal geodesics of size $n-1$. Hence, additional to the triples given in Claim \ref{clm:largegeofour}, we have that $xyz, yzx, zxy\in \betq$ for $x \in X$, $y \in Y$ and $z \in Z$. Furthermore, for any $V \in \{X,Y,Z \}$ and $s= \lvert V \rvert$ the sequences $u v_s \ldots v_1$ and $v_s \ldots v_1 u$ are maximal geodesics. So, we conclude that for $a,b \in Q$
    \begin{equation*}
        \overrightarrow{ab} = 
        \begin{cases}
            \ell_x &\text{ if } b \ordx a \text{ or } \{ a,b \} \cap X \neq \emptyset \text{ and } u \in \{ a,b \}, \\
            \ell_y &\text{ if } b \ordy a \text{ or } \{ a,b \} \cap Y \neq \emptyset \text{ and } u \in \{ a,b \}, \\
            \ell_z &\text{ if } b \ordz a \text{ or } \{ a,b \} \cap Z \neq \emptyset \text{ and } u \in \{ a,b \}, \\
            \ell &\text{ otherwise. }
        \end{cases}
    \end{equation*}
    and also $\betq \cong \mathcal{D}_1^{p,q,r}$.

    Now, consider the case $\rho (z_r,x_1) \geq \rho (z_r,y_1)$. Note that $\rho (z_r,x_1) \neq \rho (z_r,y_1)$, as otherwise we get the contradiction $y_1\notin \overrightarrow{z_rx_1}$, since $y_1 x_1 z_r \in \betq$. So, it follows that $\rho (z_r,x_1) > \rho (z_r,y_1)$. Since $x_1 \in \overrightarrow{z_ry_1}$ and $\rho (x_1,z_r) > \rho(x_1,y_1)$, then $z_1 \ldots z_r y_1 \ldots y_q x_1 \ldots x_p$ is also a maximal geodesic with $n-1$ points. Moreover, there is no other maximal geodesic with $n-1$ points. Hence, additional to the triples in $\betq$ given in Claim \ref{clm:largegeofour} we also have that $xyz, zyx\in \betq$ for $x\in X$, $y\in Y$ and $z\in Z$. So, $\betq \cong \mathcal{D}_2^{p,q,r}$. 
\end{proof}
\begin{corollary}
For all $n\geq 5$ there are $2p_3(n-1)$ non isomorphic betweennesses arising from quasimetric spaces on $n$ points without universal lines and with exactly four lines.
\end{corollary}

\section{Conclusion}
In this work, we have shown that the minimum number of lines in a quasi-metric space with $n \geq 4$ points and without universal lines is 3, and that this bound is tight, for all $n \geq 4$. This shows a structural distinction between metric and quasimetric spaces without universal lines in term of their betweennesses. Moreover, we have classified  the betweennesses of those quasimetric spaces achieving this bound. The same has been done for quasimetric spaces on $n \geq 5$ points without universal lines and exactly four lines. A natural question arising from our results is: Which are the betweennesses arising from quasimetric spaces on at least  six points, without universal lines and with exactly five lines? Can these spaces have no symmetric pairs?

In order to answer these questions, we think that it is worth to better understand the role of the symmetric sets in quasimetric spaces without universal lines. For instance, Lemma \ref{l:nonsymmetric} shows that all quasimetric spaces with 3 or 4 lines (none of which is universal) on $n \geq 5$ points have a symmetric pair. In the same way, Proposition \ref{p:nonthreesymmetric} shows that all quasimetric spaces with exactly 4 lines (none of which are universal) have a symmetric triple. These two results, together with the constructions from Theorems \ref{thm:existence} and \ref{thm:exisfour}, show the role of symmetric sets in the construction of quasimetric spaces with few lines. It is also important to note that these quasimetric spaces can be partitioned into simple metric spaces: the betweenneses of the metric spaces induced by $X, Y$ and $Z$ are isomorphic to that of a set of collinear points of the Euclidean plane.

\section*{Acknowledgements}

This work was carried out while the author Guillermo Gamboa Quintero was a participant in the CMM-PhD Visiting Program 2023 at the Center for Mathematical Modeling of the Universidad de Chile, supported by ANID Basal program FB210005.

\printbibliography

@article{huang,
      title={Improved lower bound towards {C}hen-{C}hv\'atal conjecture}, 
      author={Huang, C.},
      year={2023},
      eprint={2310.15058},
      archivePrefix={arXiv},
      primaryClass={math.CO},
      journal={\sc arXiv: 2310.15058 [math.co]},
}

@article{menger,
author = {Menger, K.},
journal = {Mathematische Annalen},
pages = {75-163},
title = {Untersuchungen über allgemeine {M}etrik},
url = {http://eudml.org/doc/159284},
volume = {100},
year = {1928},
}

@article {sylvester,
    AUTHOR = {Sylvester, J.J.},
     TITLE = {Mathematical question 11851},
   JOURNAL = {Educ. Times},
  FJOURNAL = {},
    VOLUME = {59},
      YEAR = {1893},
     PAGES = {98-99},
      ISSN = {},
   MRCLASS = {},
  MRNUMBER = {},
MRREVIEWER = {},
}

@article {gallai,
    AUTHOR = {Gallai, T.},
     TITLE = {Graphen mit triangulierbaren ungeraden {V}ielecken},
   JOURNAL = {Magyar Tud. Akad. Mat. Kutat\'{o} Int. K\"{o}zl.},
  FJOURNAL = {A Magyar Tudom\'{a}nyos Akad\'{e}mia. Matematikai Kutat\'{o}
              Int\'{e}zet\'{e}nek K\"{o}zlem\'{e}nyei},
    VOLUME = {7},
      YEAR = {1962},
     PAGES = {3--36},
      ISSN = {0541-9514},
   MRCLASS = {55.10},
  MRNUMBER = {145508},
MRREVIEWER = {J.\ W.\ Moon},
}

@article {erdos,
    AUTHOR = {Erdős, P.},
     TITLE = {Three Collinear points in {P}roblems and {S}olutions: {A}dvanced {P}roblems: {P}roblems
              for {S}olution: 4065-4069},
   JOURNAL = {Amer. Math. Monthly},
  FJOURNAL = {American Mathematical Monthly},
    VOLUME = {50},
      YEAR = {1943},
    NUMBER = {1},
     PAGES = {65--66},
      ISSN = {0002-9890,1930-0972},
   MRCLASS = {99-04},
  MRNUMBER = {1543984},
       DOI = {10.2307/2304011},
       URL = {https://doi.org/10.2307/2304011},
}

@article {aramat,
    AUTHOR = {Araujo-Pardo, G. and Matamala, M.},
     TITLE = {Chen and {C}hv\'{a}tal's conjecture in tournaments},
   JOURNAL = {European J. Combin.},
  FJOURNAL = {European Journal of Combinatorics},
    VOLUME = {97},
      YEAR = {2021},
     PAGES = {Paper No. 103374, 8},
      ISSN = {0195-6698,1095-9971},
   MRCLASS = {05C20 (05C12 05C38 05C40)},
  MRNUMBER = {4275622},
MRREVIEWER = {M.\ E.\ Watkins},
       DOI = {10.1016/j.ejc.2021.103374},
       URL = {https://doi.org/10.1016/j.ejc.2021.103374},
}

@article {chvatal2004,
    AUTHOR = {Chv\'{a}tal, V.},
     TITLE = {Sylvester-{G}allai theorem and metric betweenness},
   JOURNAL = {Discrete Comput. Geom.},
  FJOURNAL = {Discrete \& Computational Geometry. An International Journal
              of Mathematics and Computer Science},
    VOLUME = {31},
      YEAR = {2004},
    NUMBER = {2},
     PAGES = {175--195},
      ISSN = {0179-5376,1432-0444},
   MRCLASS = {52C10 (52A37)},
  MRNUMBER = {2060634},
       DOI = {10.1007/s00454-003-0795-6},
       URL = {https://doi.org/10.1007/s00454-003-0795-6},
}

@article {chichv,
    AUTHOR = {Chiniforooshan, E. and Chv\'{a}tal, V.},
     TITLE = {A de {B}ruijn-{E}rdős theorem and metric spaces},
   JOURNAL = {Discrete Math. Theor. Comput. Sci.},
  FJOURNAL = {Discrete Mathematics \& Theoretical Computer Science. DMTCS.},
    VOLUME = {13},
      YEAR = {2011},
    NUMBER = {1},
     PAGES = {67--74},
      ISSN = {1365-8050},
   MRCLASS = {52C10 (54E35)},
  MRNUMBER = {2812604},
MRREVIEWER = {T.\ Thrivikraman},
}

@article{aboulker,
  title={Lines, Betweenness and Metric Spaces},
  author={Aboulker, P. and Chen, X. and Huzhang, G. and Kapadia, R. and Supko, C.},
  journal={Discrete \& Computational Geometry},
  year={2014},
  volume={56},
  pages={427-448},
  url={https://api.semanticscholar.org/CorpusID:29317265},
}

@article {dbe,
    AUTHOR = {de Bruijn, N.G. and Erdős, P.},
     TITLE = {On a combinatorial problem},
   JOURNAL = {Nederl. Akad. Wetensch., Proc.},
  FJOURNAL = {Proceedings. Akadamie van Wetenschappen Amsterdam.              North-Holland, Amsterdam},
    VOLUME = {51},
      YEAR = {1948},
     PAGES = {1277--1279 = Indagationes Math. {\bf 10}, 421--423},
      ISSN = {0370-0348},
   MRCLASS = {09.0X},
  MRNUMBER = {28289},
MRREVIEWER = {M.\ Hall, Jr.},
}

@article{aramatzam,
title = {A de {B}ruijn and {E}rdős property in quasi-metric spaces with four points},
journal = {Procedia Computer Science},
volume = {223},
pages = {308-315},
year = {2023},
note = {XII Latin-American Algorithms, Graphs and Optimization Symposium (LAGOS 2023)},
issn = {1877-0509},
doi = {https://doi.org/10.1016/j.procs.2023.08.242},
url = {https://www.sciencedirect.com/science/article/pii/S1877050923010153},
author = {Araujo-Pardo, G. and Matamala, M. and Zamora, J.}
}

@article{chenchvatal,
	author = {Chen, X. and Chv\'{a}tal, V.},
	doi = {10.1016/j.dam.2007.05.036},
	issn = {0166-218X},
	journal = {Discrete Applied Mathematics},
	note = {In Memory of Leonid Khachiyan (1952 - 2005 )},
	number = {11},
	pages = {2101–2108},
	title = {Problems related to a de {B}ruijn–{E}rdős theorem},
	url = {https://www.sciencedirect.com/science/article/pii/S0166218X07001710},
	volume = {156},
	year = {2008}
}

@article{abbemaza,
author = {Aboulker, P. and Beaudou, L. and Matamala, M. and Zamora, J.},
title = {Graphs with no induced house nor induced hole have the de {B}ruijn–{E}rdős property},
journal = {Journal of Graph Theory},
volume = {100},
number = {4},
pages = {638-652},
doi = {https://doi.org/10.1002/jgt.22799},
url = {https://onlinelibrary.wiley.com/doi/abs/10.1002/jgt.22799},
eprint = {https://onlinelibrary.wiley.com/doi/pdf/10.1002/jgt.22799},
year = {2022}
}

@article{abkha,
	author = {Aboulker, P. and Kapadia, R.},
	doi = {10.1016/j.ejc.2014.06.009},
	issn = {0195-6698},
	journal = {European Journal of Combinatorics},
	pages = {1–7},
	title = {The {C}hen–{C}hv\'{a}tal conjecture for metric spaces induced by distance-hereditary graphs},
	url = {https://www.sciencedirect.com/science/article/pii/S0195669814000997},
	volume = {43},
	year = {2015}
}

@article{beboch,
  TITLE = {{A De {B}ruijn--{E}rdős theorem for chordal graphs}},
  AUTHOR = {Beaudou, L. and Bondy, A. and Chen, X. and Chiniforooshan, E. and Chudnovsky, M. and Chv\'{a}tal, V. and Fraiman, N. and Zwols, Y.},
  URL = {https://hal.sorbonne-universite.fr/hal-01263335},
  JOURNAL = {{The Electronic Journal of Combinatorics}},
  PUBLISHER = {{Open Journal Systems}},
  VOLUME = {22},
  NUMBER = {1},
  PAGES = {1.70},
  YEAR = {2015},
  MONTH = Mar,
  HAL_ID = {hal-01263335},
  HAL_VERSION = {v1},
}

@article{bekharo,
  TITLE = {{Bisplit graphs satisfy the {C}hen-{C}hv\'{a}tal conjecture}},
  AUTHOR = {Beaudou, L. and Kahn, G. and Rosenfeld, M.},
  URL = {https://dmtcs.episciences.org/5509},
  DOI = {10.23638/DMTCS-21-1-5},
  JOURNAL = {{Discrete Mathematics \& Theoretical Computer Science}},
  VOLUME = {{vol. 21 no. 1, ICGT 2018}},
  YEAR = {2019},
  MONTH = May,
}

@article{ch2014,
    author = {Chv\'{a}tal, V.},
    journal = {Czechoslovak Mathematical Journal},
    language = {eng},
    number = {1},
    pages = {45-51},
    publisher = {Institute of Mathematics, Academy of Sciences of the Czech Republic},
    title = {A de {B}ruijn-{E}rdős theorem for $1$-$2$ metric spaces},
    url = {http://eudml.org/doc/261986},
    volume = {64},
    year = {2014},
}

@article{jualrove,
    title = {Solution of the {C}hen-{C}hv\'{a}tal conjecture for specific classes of metric spaces},
    journal = {AIMS Mathematics},
    volume = {6},
    number = {7},
    pages = {7766-7781},
    year = {2021},
    issn = {2473-6988},
    doi = {10.3934/math.2021452},
    url = {https://www.aimspress.com/article/doi/10.3934/math.2021452},
    author = {Rodríguez-Velázquez, J.A.},
}

@article{kantor,
author = {Kantor, I.},
title = {Lines in the {P}lane with the {$\ell_1$} {M}etric},
year = {2022},
issue_date = {Oct 2023},
publisher = {Springer-Verlag},
address = {Berlin, Heidelberg},
volume = {70},
number = {3},
issn = {0179-5376},
url = {https://doi.org/10.1007/s00454-022-00443-3},
doi = {10.1007/s00454-022-00443-3},
journal = {Discrete Comput. Geom.},
month = {Oct},
pages = {960–974},
numpages = {15},
}

\end{document}